\documentclass[a4paper,11pt]{article}

\usepackage{amsthm,amsmath}
\usepackage[%
bookmarks=true,
ocgcolorlinks,%
linkcolor=blue,%
filecolor=blue,%
citecolor=blue,%
urlcolor=blue]{hyperref}
\usepackage{amssymb}
\usepackage{mathtools}
\usepackage{graphicx}
\usepackage{psfrag}
\usepackage{verbatim}
\newcommand{\bbM}{\mathbb{M}}
\newcommand{\bbE}{\mathbb{E}}
\newcommand{\bbV}{\mathbb{V}}

\newcommand{\Var}{\bbV{\rm ar}}
\newcommand{\bbP}{\mathbb{P}}

\newcommand{\bbR}{\mathbb{R}}

\newcommand{\wh}{\widehat}

\newcommand{\argmax}{\text{argmax}}

\newcommand{\cA}{\mathcal A}

\newcommand{\cC}{\mathcal C}

\newcommand{\cW}{\mathcal W}

\newcommand{\cE}{\mathcal E}
\newcommand{\cN}{\mathcal N}
\newcommand{\cF}{\mathcal F}
\newcommand{\cM}{\mathcal M}

\newcommand{\rmc}{{\rm c}}

\newcommand{\rmd}{{\rm d}}
\newcommand{\rme}{{\rm e}}
\newcommand{\rmB}{{\rm B}}

\newcommand{\wt}{\widetilde}

\newtheorem{theorem}{Theorem}[section]
\newtheorem{corollary}[theorem]{Corollary}
\newtheorem{lem}[theorem]{Lemma}
\newtheorem{prop}[theorem]{Proposition}
\newtheorem{thm}[theorem]{Theorem}
\newtheorem{rem}[theorem]{Remark}

\numberwithin{equation}{section}

\title{The Structure of Extreme Level Sets in Branching Brownian Motion}

\author
{Aser Cortines\thanks{aser.cortinespeixoto@math.uzh.ch, lhartung@uni-mainz.de,  oren.louidor@gmail.com.} \\Universit\"at Z\"urich \and Lisa Hartung\footnotemark[1]\\ 
Universit\"at Mainz\and Oren Louidor\footnotemark[1]\\Technion, Israel}
\date{}

\begin{document}

\maketitle

\begin{abstract}
We study the structure of extreme level sets of a standard one dimensional branching Brownian motion, namely the sets of particles whose height is within a fixed distance from the order of the global maximum. It is well known that such particles congregate at large times in clusters of order-one genealogical diameter around local maxima which form a Cox process in the limit. We add to these results by finding the asymptotic size of extreme level sets and the typical height of the local maxima whose clusters carry such level sets. We also find the right tail decay of the distribution of the distance between the two highest particles. These results confirm two conjectures of Brunet and Derrida~\cite{brunet2011branching}. The proofs rely on a careful study of the cluster distribution.
\end{abstract}

\setcounter{tocdepth}{2}
\tableofcontents


\section{Introduction and Results}
\subsection{Introduction}
\label{ss:Introduction}
This work concerns the fine structure of extreme values of branching Brownian motion. The latter describes the motion of a particle which diffuses on the real line according to a standard Brownian motion for a time whose law is exponential with mean one and then splits into two independent child particles which repeat the same procedure starting from the last position of their parent.

One way of formulating this process is as follows. Take a continuous time (binary) Galton-Watson tree $T=(T_t :\: t \geq 0)$ with branching rate $1$ and denote by $L_t$ its set of leaves at time $t$, so that $\bbE |L_t| = \rme^t$. Then conditional on $T$, let $h = (h_t(x) :\: t \geq 0,\, x \in L_t)$ be a mean-zero Gaussian process with covariance function given by
\begin{equation}
\bbE h_t(x) h_{t'}(x') = \sup \{s \geq 0 :\: x, x' \text{ share a common ancestor in } L_s\},
\end{equation}
where $t , t' \geq 0$ and $x \in L_t,\, x' \in L_{t'}$. The connection with the description above is then obtained by interpreting $L_t$ as the set of particles alive at time $t$ and $h_t(x)$ as the position of particle $x \in L_t$.

The study of extreme values of $h$ dates back to works of Ikeda et al.~\cite{Watanabe1,Watanabe2,Watanabe3}, McKean~\cite{McKean}, Bramson \cite{B_C, Bramson1978maximal} and Lalley and Sellke~\cite{LS} who derived asymptotics for the law of the
maximal height $h^*_t = \max_{x \in L_t} h_t(x)$. Introducing the centering function
\begin{equation}
\label{eq_m_t}
m_t := \sqrt{2}t - \frac{3}{2\sqrt{2}} \log^+ t \,,
\quad  \text{where } \qquad
\log^+ t := \log (t \vee 1) \,,
\end{equation}
and writing
$\wh{h}_t$ for the centered process $h_t - m_t$ and $\wh{h}^*_t : = h_t^\ast - m_t$ for its maximum, they show that
$\wh{h}^*_t$ converges in law to $G + \log Z$ as $t \to \infty$,
where $G$ is a Gumbel distributed random variable and $Z$, which is independent of $G$, is the almost sure limit as $t \to \infty$ of (a multiple of) the so-called {\em derivative martingale}:
\begin{equation}
\label{e:303}
\textstyle
Z_t := C_\diamond \sum_{x \in L_t} \big( \sqrt{2} t - h_t(x) \big) \rme^{\sqrt{2} (h_t(x) - \sqrt{2}t)} \,,
\end{equation}
for some $C_\diamond > 0$ properly chosen. Henceforth we use this unconventional normalization, to avoid carrying the constant $C_\diamond$ around in all occurrences of $Z$.

Other extreme values of $h$ can be studied simultaneously by considering the extremal process:
\begin{equation}
\label{e:0.1}
\textstyle 
\cE_t := \sum_{x \in L_t} \delta_{h_t(x) - m_t} \,.
\end{equation}
Asymptotics for this process were treated in the physics literature by, e.g., Brunet and Derrida~\cite{brunet2011branching} and more recently in the mathematical literature simultaneously by A\"id\'ekon et al.~\cite{ABBS} and Arguin et al.~\cite{ABK_E}. These works show that there exists a random point measure $\cE$ such that
\begin{equation}
\label{e:O.2}
\cE_t \Longrightarrow \cE
\quad \text{as } t \to \infty \,,
\end{equation}
in the sense of weak convergence of distributions on the space $\bbM$ of Radon measures on $\bbR$ endowed with the vague topology. The process $\cE$ turns out to be a randomly shifted clustered Poisson point process (PPP) with an exponential intensity. More explicitly, there exists a non-degenerate {\em cluster distribution} $\nu$ on the set of point measures in $\bbM$ with support in $(-\infty, 0]$, such that $\cE$ can be realized as
\begin{equation}
\label{e:5}
\textstyle
\cE: = \sum_{k \geq 1} \cC^k (\cdot - u^k),
\end{equation}
where $\big(\cC^k :\: k \geq 1 \big)$ are independently chosen according to $\nu$ and the ordered sequence $u^1>u^2>\dots$ forms the atoms of the point process $\cE^*$, whose law is determined via
\begin{equation}
\label{e:5.5}
\cE^* | Z \sim \text{PPP}\big(Z \rme^{-\sqrt{2}u} \rmd u\big) \,,
\end{equation}
with $Z$ defined as above.

For what follows in the paper we shall use a slightly stronger version of the convergence in \eqref{e:O.2}. To state it, let us first endow the set $L_t$ with the genealogical distance $\rmd = \rmd_t$ given by
\begin{equation}
\rmd (x,x') := \inf \{s \geq 0 :\: x, x' \text{ share a common ancestor in } L_{t-s}\}\ , \ \ 
\end{equation}
where $t \geq 0$ and $x,x' \in L_t$.
Then, given $x \in L_t$ and $r > 0$, we let $\cC_{t,r}(x)$ denote the (finite time, finite diameter) cluster of relative particle heights, at genealogical distance at most $r$ from $x$, defined formally as
\begin{equation}
\label{e:5B}
\textstyle
\cC_{t,r}(x) := \sum_{y \in \rmB_{r}(x)} \delta_{h_t(y) - h_t(x)},\  \text{ where} \quad
\rmB_{r}(x) := \{y \in L_t :\: \rmd (x,y) < r\} .
\end{equation}
Finally, fixing any positive function $t \mapsto r_t$ such that both $r_t$ and $t/r_t$ tend to $\infty$ as $t \to \infty$ and letting $L_t^* = \big\{x \in L_t :\: h_t(x) \geq h_t(y) \,, \forall y \in \rmB_{r_t} (x)\big\}$, we can define the \emph{generalized extremal process} $\wh{\cE}_t$ as
\begin{equation}
\label{e:N6}
\textstyle
\wh{\cE}_t := \sum_{x \in L_t^*} \delta_{h_t(x) - m_t} \otimes \delta_{\cC_{t,r_t}(x)} \,.
\end{equation}
The process $\wh{\cE_t}$, which is a random point measure on $\bbR \times \bbM$, records both the centered height of $r_t$-local maxima of $h$ and the cluster around them.

Then the proof of Theorem~2.3 in~\cite{ABK_E} readily shows that 
\begin{equation}
\label{e:N7}
\big(\wh{\cE}_t, Z_t \big) \overset{t \to \infty}{\Longrightarrow} \big(\wh{\cE}, Z\big)\,,
\quad \text{with}\qquad 
\wh{\cE}|Z \sim {\rm PPP}(Z\rme^{-\sqrt{2}u} \rmd u \otimes \nu) \,,
\end{equation}
and $Z_t$, $Z$ and $\nu$ as before. In fact, one can realize $\cE$, $\cE^*$ and $\wh{\cE}$ 
on the same probability space such that 
\begin{equation}
\label{e:N8}
\textstyle
\cE^* = \sum_{(u, \cC) \in \wh{\cE}} \, \delta_u 
 \quad \text{and} \qquad
\cE = \sum_{(u, \cC) \in \wh{\cE}} \, \cC(\cdot - u) \,, 
\end{equation}
with the sums running over all points in the support of $\wh{\cE}$. Moreover, letting 
\begin{equation}
\label{e:M1.13}
\textstyle 
\cE_t^* := \sum_{x \in L^*_t} \delta_{h_t(x) - m_t} \,,
\end{equation}
we clearly have $\cE^*_t \Longrightarrow \cE^*$ as $t \to \infty$.

This explains the clustered structure of the limit process $\cE$ as given by~\eqref{e:5}. The ``back-bone'' Poisson point process $\cE^*$ captures the asymptotics of extreme values which are also the local maxima in an $O(1)$-genealogical neighborhoods around them, while the clusters $(\cC^k :\: k\geq 1)$ describe the asymptotic law of the (relative) heights of particles in these neighborhoods. 

The validity of this description, or equivalently of relation~\eqref{e:N8} is a consequence of the following result from~\cite{ABK_G} (Theorem 2.1), which shows that particles achieving extreme height separate in the limit into clusters of diameter $O(1)$ which are $t-O(1)$ apart (in genealogical distance), namely:
\begin{equation}\label{e:108}
\varlimsup_{\substack{t \to \infty \\ r \to \infty}}
\bbP\Big( \exists  x , y \! \in \! L_t :\: h_t(x) \wedge h_t(y) \! > \! m_t \!+ v \text{ and } \rmd(x,y)\in[r,t\!- r]  \Big)\! = \!0 \,,
\end{equation} 
for all $v \in \bbR$, where $r \to \infty$ after $t \to \infty$ in the limit superior.

Naturally, the clustered structure of $\cE$ implies that its structural features will be determined by the properties of the cluster distribution $\nu$. Two different albeit equivalent descriptions of the latter have been given in~\cite{ABBS} and~\cite{ABK_E}. In~\cite{ABK_E} (Theorem 2.1, Proposition 2.9) it is described as the $t \to \infty$ limit of the configuration of heights seen from the maximal particle, when the latter is conditioned to reach the unlikely height of  $\sqrt{2}t$. The existence of this limit was first shown by Chauvin and Rouault~\cite{ChaRou1990} who described it in terms of a distinguished ``spine'' particle (see Subsection~\ref{ss:Spine}) which produces offspring at an increased rate and reaches the unusual height.  Alternative descriptions of $\nu$ are given in~\cite{ABBS} (Theorem 2.3 and Theorem 2.4) in terms of a distinguished particle moving according to a Brownian motion in a potential, from which branching Brownian motions descend and are conditioned to stay above zero.

\subsection{Results}
\label{s:Results}
In this manuscript we provide a more detailed description of the extreme level sets of branching Brownian motion, improving upon the state-of-the-art as outlined above (see also Subsection~\ref{ss:Discussion}). The term {\em extreme (super/upper) level set} will be used in this work to refer to the set of indices or heights of particles in $L_t$ whose value under $h_t$ is above $m_t + v$ for some fixed $v \in \bbR$. In light of convergence statements~\eqref{e:O.2} and~\eqref{e:N7}, such results can be stated, rather equivalently, both in an asymptotic form or directly in terms of the limiting objects. Since each form is of interest by itself, we will use both formulations. 

In what follows, we say that $f(u,v)$ converges to $F$ in the limit when $u \to u_0$ followed by $v \to v_0$, to mean that $\lim_{v \to v_0} \limsup_{u \to u_0} |f(u,v) - F| = 0$. If $f(u,v)=f_w(u,v)$ and $F=F_w$, then this converges is uniform in $w \in \cW$, if the above holds with an additional $\sup_{w \in \cW}$ before the absolute value. We write $f(u) \sim g(u)$ as $u \to u_0$ to mean that $f(u)/g(u) \to 1$ as $u \to u_0$. This should not be confused with the notation for ``is distributed according to'' which will use the same symbol. Finally, arbitrary positive constants are marked by decorated version of the letter $C$ (e.g. $C'$) and unless otherwise specified, they may change their value from one line to another.

\subsubsection{Extreme Level Sets}
\label{ss:Global}
Our first result concerns the asymptotic size of the level set of extreme values at height $m_t - v$. The following theorem confirms a conjecture by Brunet and Derrida (Subsection 4.3 in \cite{brunet2011branching}, see also Subsection~\ref{ss:Discussion} below). 
\begin{thm}\label{thm.asyden}
There exists $C_\star > 0$ such that 
\begin{equation}
\label{thm.asyden_eq1}
\frac{\cE([-v, \infty))}{C_\star Z v \rme^{\sqrt{2} v}} \overset{\bbP}{\longrightarrow} 1
\quad \text{as } v \to \infty \,.
\end{equation}
In particular, for all $\epsilon > 0$,
\begin{equation}
\label{thm.asyden_eq2}
	\lim_{v \to \infty} \limsup_{t \to \infty} 
		\bbP \left( \left | \frac{\cE_t([-v, \infty))}{C_\star Z v \rme^{\sqrt{2} v}} - 1 \right| > \epsilon \right) = 0 \,.
\end{equation}
\end{thm}

The asymptotic growth rate (as $v \to \infty$) of the number of points in $\cE$ should be compared with the growth rate of the number of points in the process $\cE^*$, which records the limit of only those extreme values which are also local maxima. It follows from~\eqref{e:5.5} and 
a simple application of the weak law of large numbers that
\begin{equation}
\label{e:120}
\begin{split}
&\frac{\cE^*([-u, \infty))}{Z \rme^{\sqrt{2} u}/\sqrt{2}} 
	\overset{\bbP}{\underset{u \to \infty} \longrightarrow}  1 \\
& \qquad \text{and}  \qquad
\lim_{u \to \infty}\limsup\limits_{t \to \infty}
		\bbP \left( \left | \frac{\cE^*_t([-u, \infty))}{Z \rme^{\sqrt{2} u}/\sqrt{2}} - 1 \right| > \epsilon \right) = 0 \,.
\end{split}
\end{equation}
The above shows that points coming from the clusters around extreme local maxima account for an additional multiplicative linear prefactor in the overall growth rate of extreme values.

This gives rise to the following natural question: What is the ``typical'' height of those local maxima in $\cE^*_t|_{[-v,\infty)}$ whose cluster points ``carry'' the level set $\cE_t|_{[-v, \infty)}$? As the next theorem shows, the contribution is essentially uniform across all heights in $[-v, \infty)$. For a precise statement, recall~\eqref{e:N8}, then given a Borel set $B \subseteq \bbR$ define,
\begin{equation}
\label{e:421}
\begin{split}
\cE (\cdot \;; B) &:= \sum\limits_{(u,\cC) \in \wh{\cE}} \cC ( \cdot - u) 1_{\{u \in B\}} \\
& \quad \text{ and} \quad 
\cE_t (\cdot \;; B) := \sum\limits_{(u,\cC) \in \wh{\cE}_t} \cC ( \cdot - u) 1_{\{u \in B\}} \,.
\end{split}
\end{equation}
Then,
\begin{thm}
\label{t:14}
Fix any $\alpha \in (0,1]$. Then as $v \to \infty$, 
\begin{equation}
\label{e:37}
\frac{\cE\big([-v,\infty) ;\; [-\alpha v, \infty) \big)}{
\cE \big([-v,\infty)\big)} \overset{\bbP}{\longrightarrow} \alpha \,.
\end{equation}
In particular, 
\begin{equation}
\label{e:624}
\lim_{v \to \infty}\limsup_{t \to \infty}
	\bbP \left( \left| \frac{\cE_t\big([-v,\infty) ;\; [-\alpha v, \infty) \big)}{
	\cE_t \big([-v,\infty)\big)} - \alpha \right| > \epsilon \right) = 0 \,.
\end{equation}
\end{thm}

We can rephrase the statement in \eqref{e:624} in terms of a uniform sampling from all particles whose height is above $m_t -v$ as follows:
\begin{corollary}
\label{c:1}
Given $t, v > 0$, let $X$ be a particle chosen uniformly from all particles $x \in L_t$ satisfying $\wh{h}_t(x) \geq - v$ and set $Y := \argmax \big\{ \wh{h}_t(y) : \: y \in \rmB_{t,r_t}(X) \big\}$. Then as $t \to \infty$ followed by $v \to \infty$,
\begin{equation}
\frac{\wh{h}_t(Y) - (-v)}{v} \Longrightarrow  U([0,1]) \,.
\end{equation}
\end{corollary}
Roughly speaking, for each $u \in [O(1), v]$ the total contribution to the level set $\cE_t\big([-v, \infty)\big)$ from clusters around local maxima at height $m_t - u$ is uniformly $\sim C_\star Z \rme^{\sqrt{2}v}$, making the total size of the level set $\sim C_\star Z v\rme^{\sqrt{2}v}$ in agreement with Theorem~\ref{thm.asyden}. 

Lastly, we find the rate of decay of the right tail probabilities of the distance between the maximum and the second maximum particles in $h_t$, thereby confirming another conjecture of Brunet and Derrida (Subsection~4.2 in~\cite{brunet2011branching}). Setting $h_t^{*(2)} := \max \big\{h_t(x) :\: x \in L_t ,\, h_t(x) < h^*_t \big\}$, we have
\begin{thm}
\label{t:2.9}
Let $v^1 > v^2 > \dots$ be the ordered atoms of $\cE$. Then
\begin{equation}\label{equation.right.tail.statistics.BBM}
	\lim_{w \to \infty} w^{-1} \log \bbP \big( v^1 - v^2 > w) = -(2+\sqrt{2}) \,.
\end{equation}
In particular, 
\begin{equation}
\label{e:33}
\lim_{w \to \infty} \limsup_{t \to \infty} \Big| w^{-1} \log \bbP \big(  h^*_t - h_t^{*(2)} > w \big) + (2+\sqrt{2}) \Big| = 0 \,.
\end{equation}
\end{thm}

\subsubsection{Cluster Level Sets}
As evident by~\eqref{e:5}, the key to obtaining the theorems above lies in obtaining corresponding structural results concerning the cluster distribution $\nu$. Thanks to a good control over the convergences in~\eqref{e:O.2},~\eqref{e:N7} and the explicit description of $\cE$ and $\wh{\cE}$, one can turn local asymptotic properties of clusters into global statements concerning these limit processes, and then to asymptotic results for the extreme level sets of $h_t$ itself. In this subsection we therefore state the cluster law properties, which are used to derive the main theorems in this paper. These properties should be of independent interest.

\label{ss:ClusterProperties}
The first proposition concerns the asymptotic mean number of cluster particles at height $-v$ or above, as well as an upper bound on its second moment. Recall that by definition and~\eqref{e:N7}, if $\cC \sim \nu$
then $\cC([0,\infty)) = \cC(\{0\}) = 1$ almost-surely.
\begin{prop}
\label{p:12}
Let $\cC \sim \nu$. Then with $C_\star > 0$ as in Theorem~\ref{thm.asyden},
\begin{equation}
\label{e:29}
\bbE \cC([-v, 0]) \sim C_\star \rme^{\sqrt{2} v} 
\ \text{as } v \to \infty \,.
\end{equation}
Moreover, there exists $C > 0$ such that for all $v \geq 0$,
\begin{equation}
\label{e:28}
\bbE \big[ \cC([-v, 0])^2 \big] \leq C (v+1) \rme^{2\sqrt{2} v} \,.
\end{equation} 
\end{prop}  
As surmised by the above upper bound, the number of points in $\cC$ lying in $[-v,0]$ does not concentrate around its mean for large $v$. 

In the next proposition we find the rate of decay in the right tail of the distribution of the distance between the top two cluster particles.
\begin{prop}\label{prop:right.tail.cluster}
Let $\cC \sim \nu$. Then,
\begin{equation}\label{equation.thm:right.tail.cluster.BBM} 
\lim_{v \to \infty} v^{-1} \log \bbP \big( \cC([-v, 0)) = 0  \big) 
= -2.
\end{equation}
\end{prop}

\subsection{Proof Outline}
\label{ss:ProofOutline}
Let us give a brief outline of the proof of the main results in this paper. As mentioned before, the key ingredient in deriving results pertaining to the extremal landscape of the process is the study of the cluster distribution $\nu$. Aside from the limit of the derivative martingale $Z$, whose effect is merely a global shift, all remaining ingredients in the definition of $\cE$ and $\wh{\cE}$ are explicit. Properties of the cluster law can therefore be translated via~\eqref{e:5} or~\eqref{e:N7} and~\eqref{e:N8}, to properties of $\cE$ and $\wt{\cE}$ and through convergences~\eqref{e:O.2} and~\eqref{e:N7} into asymptotic properties of the statistics of extreme values of $h$.

\subsubsection{Cluster Level Sets}
The study of cluster law properties, which constitutes the core of the paper, begins by observing that the product structure of the intensity measure in~\eqref{e:N7} and indistinguishably of particles, imply that we could focus on the limiting law of the cluster around 
a uniformly chosen particle $X_t$ in $L_t$, conditioned to be the global maximum at time $t$ and having height, say, $m_t$. Tracing the trajectory of this distinguished particle backwards in time and accounting, via the spinal decomposition (Many-to-one Lemma, see Subsection~\ref{ss:Spine}), for the random genealogical structure, one sees a particle performing 
a standard Brownian motion $W=(W_s)_{s \geq 0}$ from $m_t$ at time $0$ to $0$ at time $t$. This, so-called, {\em spine particle} gives birth at random Poissonian times (at an accelerated rate $2$, see Subsection~\ref{ss:Spine}) to independent standard branching Brownian motions, which then evolve back to time $0$ and are conditioned to have their particles stay below $m_t$ at this time. The cluster distribution at genealogical distance $r$ around $X_t$ is therefore determined by the relative heights of particles of those branching Brownian motions which branched off before time $r$ (see Figure~\ref{f:Spinal}).

\begin{figure}[ht!]
\centering
    \includegraphics[width=0.65\textwidth]{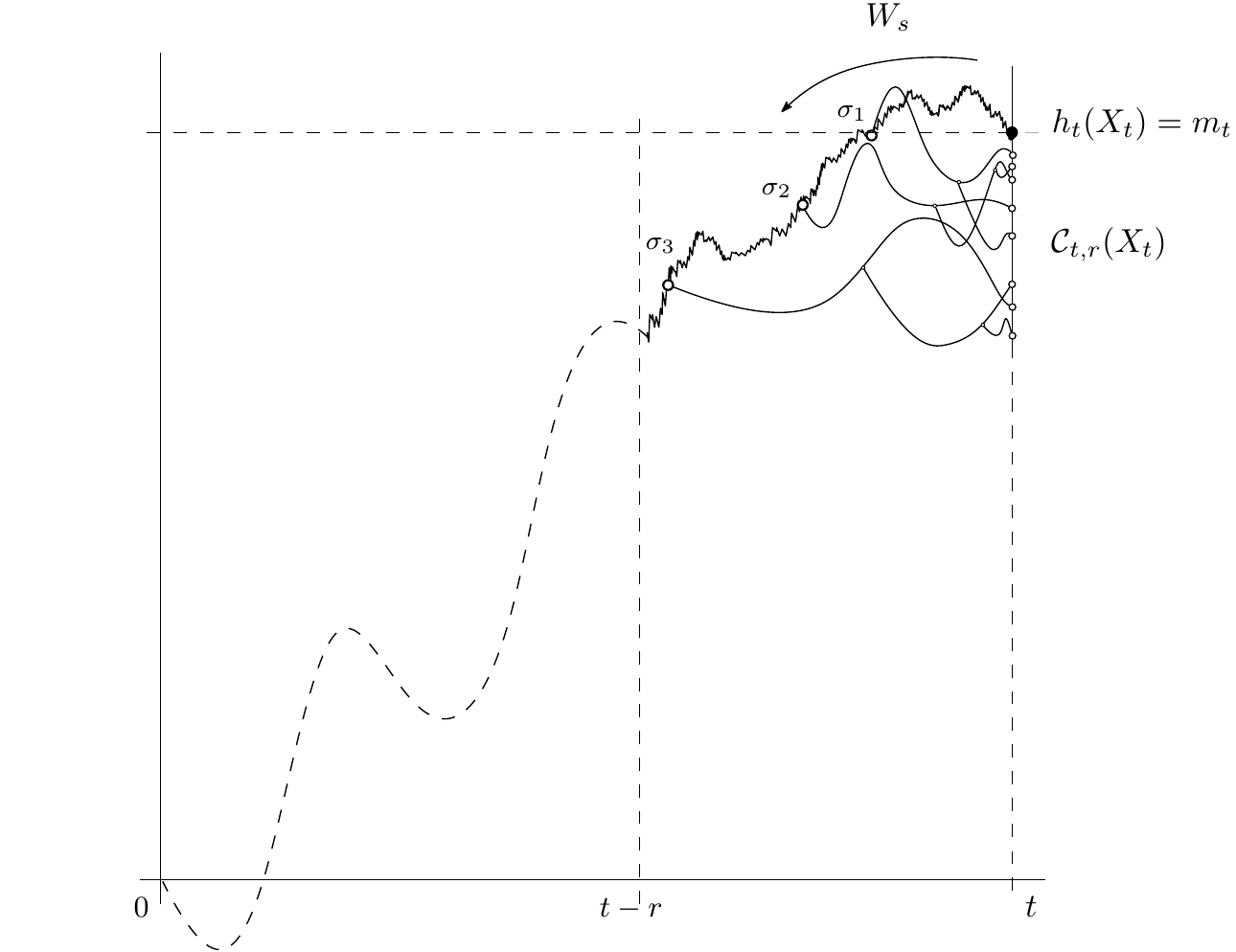} 
    \caption{The cluster $\cC_{t,r}(X_t)$ around the spine $X_t$, conditioned to be the maximum and at height $m_t$. The process $W_s$ is a Brownian bridge from $(0,m_t)$ to $(t,0)$ and $\sigma_1, \sigma_2, \dots$ are the branching times.
    } \label{f:Spinal}
\end{figure}

Formally, denoting by $0 \leq \sigma_1 < \sigma_2 < \dots$ the points of a Poisson point process $\cN$ on $\bbR_+$ with rate $2$ and letting $H = (h^{s}_t(x) :\: t \geq 0\,,\,\, x \in L^s_t)_{s \geq 0}$ be a collection of independent branching Brownian motions (with $W, \cN$ and $H$ independent),  the limiting distribution $\nu(\cdot)$ may be written (Lemma~\ref{l:7.0}) as the $t\to \infty$ limit of $\bbP \big(\cC_{t,r_t} (X_t) \in \cdot \big|\, h^*_t = h_t(X_t) = m_t \big)$, where $r_t$ is as in~\eqref{e:N6} and $h^*_t = \max_{x \in L_t} h_t(x)$. Writing further $\bbP_{0,x}^{t,y}$ for the conditional probability measure $\bbP(\cdot \,|\, W_0 = x \,,\,\, W_t = y)$, this probability reads as 
\begin{equation}
\label{e:M1.33}
\textstyle
\bbP_{0,m_t}^{t,0} \! \Big( \! \sum\limits_{\sigma_k \leq r_t} \sum\limits_{x \in L^{\sigma_k}_{\sigma_k}} \! \delta_{h^{\sigma_k}_{\sigma_k}(x)-m_t} \big(\cdot - W_{\sigma_k}\big) \! \in \! \cdot
	 \ \Big| \max\limits_{k: \sigma_k \in [0,t]}  \big(W_{\sigma_k} \! +  h^{\sigma_k *}_{\sigma_k}\big) \! \leq  \! m_t \Big).
\end{equation}

Since the law of $W_s$ under $\bbP_{0,m_t}^{t,0}$ is the same as that of $W_s + m_t (1-\frac{s}{t})$ under $\bbP_{0,0}^{t,0}$, introducing $\wh{W}_{t,s} := W_s - \gamma_{t,s}$ with $\gamma_{t,s} := 3/(2\sqrt{2})(\log^+ s - \tfrac{s}{t}\log^+ t)$, we may rewrite the above as (Lemma~\ref{l:5.2}):
\begin{equation}
\label{e:M1.35}
\textstyle
\nu(\cdot) = \lim\limits_{t \to \infty}
\bbP_{0,0}^{t,0}
\Big( \! \sum\limits_{\sigma_k \leq r_t} \, \cE_{\sigma_k}^{\sigma_k} \big(\cdot - \wh{W}_{t,\sigma_k}\big)
	\!  \in \! \cdot 
	 \ \Big| \max\limits_{k : \! \sigma_k \in [0,t]} \big(\wh{W}_{t,{\sigma_k}}  \! +   \wh{h}^{\sigma_k*}_{\sigma_k}\big) \! \leq \!  0 \Big)\,,
\end{equation}
where $\cE^{s}_t$ is the extremal process associated with $h^{s}_t$ and $\wh{h}^s_{t} = h^s_t - m_t$. The triplet $(\wh{W}, \cN, H)$ will be referred to as a {\em decorated random-walk-like} process (see Section~\ref{s:Reduction}). We remark that this characterization bares strong resemblance to the description of the cluster distribution in~\cite{ABBS}. 

The above representation can now be used to study the distribution of the size of cluster level sets as well as the law of the distance to the second highest particle in the cluster. To estimate the first moment of the size of the cluster level set, one can use~\eqref{e:M1.35}, uniform integrability and Palm calculus to express $\bbE \cC([-v, 0])$ for $\cC \sim \nu$ and any $v \geq 0$ as 
the limit when $t \to \infty$ of
\begin{equation}
\label{e:M1.34.1}
\begin{split}
\int_{0}^{r_t} 2\rmd s \int_{O(1)} & \bbE \Big( \cE^s_s \big([-v, 0] - z\big) ;\; z + \wh{h}^{s*}_s \leq 0 \Big) \\
& \ \times \frac{\bbP_{0,0}^{t,0} \Big(
\max\limits_{k : \, \sigma_k \in [0,t]} \big(\wh{W}_{t,{\sigma_k}} + \wh{h}^{\sigma_k*}_{\sigma_k}\big) \leq 0 \,\,,\, \wh{W}_{t,s} \in \rmd z\Big)}
{\bbP_{0,0}^{t,0} \Big(\max\limits_{k : \, \sigma_k \in [0,t]} \big(\wh{W}_{t,{\sigma_k}} + \wh{h}^{\sigma_k*}_{\sigma_k}\big) \leq 0 \Big)} \rmd z \,.
\end{split}
\end{equation}
Above, we have also conditioned on $\{\wh{W}_{t,s} = z\}$ for $z = O(1)$ and used the total probability formula (see Lemma~\ref{l:7.3} and the proof of Lemma~\ref{l:7.1}).

The left most term in the integrand is the first moment of the size of the (global) extreme level set of $h^s_s$, subject to a truncation event restricting the height of its global maximum. Using once again the spinal decomposition, we can express this expectation in terms of a probability involving (again) a uniformly chosen particle $X_t$ as,
\begin{equation}
\begin{split}
\label{e:M1.37}
\bbE & \Big(\cE_t\big([-v, u]\big) ;\; \wh{h}^*_t \leq u \Big) \\
& = \rme^t \bbP \big(\wh{h}_t(X_t) \in [-v, u] \,,\,\, \wh{h}^*_t \leq u \big) \\
& = \rme^t \int_{w=-v}^u \bbP \big( \wh{h}^*_t \leq u \,\big|\, \wh{h}_t(X_t) = w \big) 
	\bbP \big( \wh{h}_t(X_t) \in \rmd w \big) \,.
\end{split}
\end{equation}
where $v \leq 0$ and $u \geq -v$. As before, tracing the trajectory of the spine particle, the last conditional probability can be further expressed in terms of the decorated random-walk-like process as,
\begin{equation}
\label{e:M1.36}
\bbP \big( \wh{h}^*_t \leq u \, \big| \, \wh{h}_t(X_t) = w \big)
= \bbP_{0,w-u}^{t,-u} \Big( \max_{k: \sigma_k \in [0,t]} \big(\wh{W}_{t,\sigma_k} + \wh{h}^{\sigma_k*}_{\sigma_k}\big) \leq 0 \Big) \,.
\end{equation}

Examining~\eqref{e:M1.34.1} and~\eqref{e:M1.36}, we see that to complete the derivation we need good estimates on probabilities of the form
$\bbP_{0,x}^{t,y} \big( \max_{k: \sigma_k \in [0,t]} \big(\wh{W}_{t,\sigma_k} + \wh{h}^{\sigma_k*}_{\sigma_k}\big) \leq 0 \big)$,
namely of the event that the random-walk-like process plus its decorations stays below $0$ at random sampling times. For standard Brownian motion, the well known reflection principle gives
\begin{equation}
\bbP_{0,x}^{t,y} \Big(\max_{s \in [0,t]} W_s \leq 0 \Big) \sim \frac{2xy}{t} 
\quad \text{as } t \to \infty\,,
\end{equation}
uniformly in $x, y \leq 0$ satisfying $xy = o(t)$ and with the right hand side holding as an upper bound for all $t \geq 0$ and $x,y \leq 0$. We show (Subsection~\ref{ss:RWEstimates}) that similar estimates hold for the decorated random-walk-like process as well. This is not very surprising, as the drift function $\gamma_{t,s}$ is bounded by $1 + \log^+ (s \wedge (t-s))$ (Lemma~\ref{l:lisa} with $r=0$), the random decorations $(h^{s*}_{s} :\: s \geq 0)$ are (at least) exponentially tight (Lemma~\ref{l:TailBBMMax}) and the random sampling times $(\sigma_k :\: k \geq 1)$ arrive at a Poissonian rate.

Using such estimates in~\eqref{e:M1.36} one obtains
$\bbP \big( \wh{h}^*_t \leq u \, \big| \, \wh{h}_t(X_t) = w \big) 
\approx C (u^++1)(u-w) t^{-1}$ (in this section $\approx$ means ``roughly equals to''). This can then be used in~\eqref{e:M1.37} together with
\begin{equation}
\begin{split}
\bbP(\wh{h}_t(X_t) \in \rmd w) & = \bbP(h_t(x) - m_t \in \rmd w) \\
& = (2\pi t)^{-1/2} \rme^{-(m_t + w)^2/{2t}} \approx C t \rme^{-t} \rme^{-\sqrt{2} w - w^2/(2t)} \rmd w\,,
\end{split}
\end{equation}
to yield (Lemma~\ref{l:18})
\begin{equation}
\label{e:M1.41}
\bbE \big(\cE_t\big([-v, u]\big) ;\; \wh{h}^*_t \leq u \big) \approx (u^++1) (u+v) C \rme^{\sqrt{2} v - v^2/(2t)}.
\end{equation}

Plugging this back into the integral in~\eqref{e:M1.34.1} and estimating the probability in the denominator by $C t^{-1}$ and the probability in the numerator by $C z^2 (s(t-s))^{-1} \bbP_{0,0}^{t,0}(\wh{W}_{t,s} \in \rmd z) \approx C t^{-1} s^{-3/2} z^2$, one obtains (after integration over $z$),
\begin{equation}
\label{e:M1.42}
\begin{split}
\bbE \cC([-v, 0]) &\approx C v \rme^{\sqrt{2}v}
 \int_{s=0}^{\infty} s^{-3/2} \rme^{-v^2/(2s)} \rmd s \\
&= C \rme^{\sqrt2{v}} \int_{r=0}^{\infty} r^{-3/2} \rme^{-1/(2r)} \rmd r 
= C' \rme^{\sqrt2{v}}\,,
\end{split}
\end{equation}
which is the first part of Proposition~\ref{p:12} with $C_\star = C'$.
Similar computations, albeit more involved, can be used to obtain an upper
bound on the second moment of $\cC([-v,0])$ as in the second part of Proposition~\ref{p:12}.

\subsubsection{Extreme Level Sets}
As suggested before, we can take advantage of convergences~\eqref{e:O.2} and~\eqref{e:N7} to  prove all results for the limit processes $\cE$ and $\wh{\cE}$ first and then convert these to asymptotic statements for $h_t$, using standard weak convergence arguments for random measures. Working directly with the limiting objects has the advantage that, equipped with the needed cluster properties, their law has an explicit and rather simple form (see \eqref{e:5},~\eqref{e:N7},~\eqref{e:N8}).

Let us demonstrate this by deriving asymptotics for the size of extreme level sets (Theorem~\ref{thm.asyden}). To this end, we show that $\cE([-v, \infty))v^{-1} \rme^{-\sqrt{2}v}$ tends to $C_\star Z$  as $v \to \infty$ in probability (Lemma~\ref{l:8.1}). Using~\eqref{e:5}, we can begin by writing $\cE([-v, \infty])$ as the sum $\sum_{k \geq 1} \cC^k([-v-u^k, 0])$, with $\cC^k$, $u^k$ as in \eqref{e:5}. Ignoring terms with $u^k \notin [-v+\sqrt{\log v}, \sqrt{\log v}]$, which are negligible in the scale we consider (see proof of Lemma~\ref{l:8.1}) and denoting by $\wt{\cE}([-v, \infty))$ the sum of the remaining terms, we can condition on $Z$ and use~\eqref{e:M1.42} together with the Poisson law of $\cE^*$ to estimate $\bbE\big(\wt{\cE}([-v, \infty) \mid Z \big)$~by 
\begin{equation}
\label{e:M1.45}
\begin{multlined}
	\int_{-v+\sqrt{\log} v}^{\sqrt{\log v}}  \bbE \cC\big([-v-u,0]\big) Z \rme^{-\sqrt{2} u} \rmd \\	
	\approx \int_{-v+\sqrt{\log v}}^{\sqrt{\log v}}  C_\star \rme^{\sqrt{2}(v+u)} Z \rme^{-\sqrt{2} u} \rmd u	
	\approx C_\star Z v \rme^{\sqrt{2}v} \,.
\end{multlined}
\end{equation}
A similar computation using the second moment bound on $\bbE \cC\big([-v-u,0]\big)$ in place of the first, shows that the conditional (on $Z$) variance of $\wt{\cE}([-v, \infty))$ is at most $Cv^{-1}$ times its conditional mean. Then Chebyshev's inequality shows that $\wt{\cE}([-v, \infty))$ is concentrated around its conditional mean, which in light of~\eqref{e:M1.45} and 
$\wt{\cE}([-v, \infty)) \approx \cE([-v, \infty))$ yields the desired result.

\subsubsection{Distance to the Second Maximum}
\label{ss:Distance}
Lastly, let us discuss the upper tail decay of the law governing the distance between the first and second maxima of $h$, namely Theorem~\ref{t:2.9} and Proposition~\ref{prop:right.tail.cluster} on which the theorem relies. Again, thanks to the convergence of the extremal process, we can look at the distance between the two highest points $v^1 > v^2$ in $\cE$. Then, the clustered structure of the limit~\eqref{e:5} readily shows that these points are at least $w > 0$ apart, if and only if the distance between the two highest local maxima $u^1 > u^2$ in $\cE^*$ and the distance to the second highest particle in cluster $\cC^1$ of $u^1$ are both at least $w$. Thanks to independence, we therefore~get
\begin{equation}
\label{e:M185}
\bbP\big(v^1 - v^2 > w \big) = \bbP \big(u^1 - u^2 > w\big) \bbP\big(\cC([-w, 0)) = 0 \big) \,.
\end{equation}
The first probability on the right hand side evaluates to $C \rme^{-\sqrt{2}w}$ (see proof of Theorem~\ref{t:2.9}). This is an easy exercise in Poisson point processes, after noticing that the random shift governing the law of $\cE^*$ can be just ignored.
\begin{figure}[ht!]
\centering
    \includegraphics[width=0.5\textwidth]{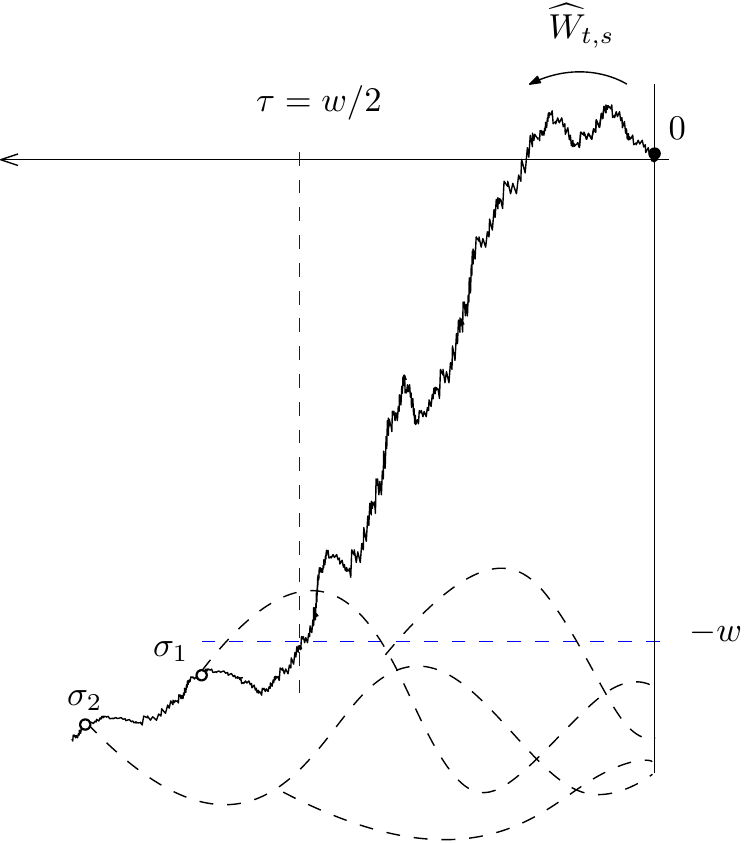} 
    \caption{A typical realization for $\{\cC([-w,0)) = 0\}$:  
    The process $\wh{W}_{t,s}$ reaches height $-w$ at time $\tau = w/2$ without branching and then, along with its decorations, stays below $-w$ until time $r_t$.}
    \label{f:distance}
\end{figure} 

For the second probability (Proposition~\ref{prop:right.tail.cluster}), we again use the random-walk representation of the cluster distribution, per~\eqref{e:M1.33} and~\eqref{e:M1.35} and estimate instead the probability that $\cC_{t,r_t}\big([-w, 0)\big) = 0$ as $t \to \infty$ under the conditional measure, where $\cC_{t,r_t} = \cC_{t,r_t}(X_t)$. For a lower bound, we
follow the heuristics of Brunet and Derrida (Subsection 4.2 in~\cite{brunet2011branching}) and observe 
that having no points in $\cC_{t,r_t}|_{[-w, 0)}$ can be realized by the intersection of the event that $\wh{W}_{t,s}$ reaches height $-w$ or below at some time $s=\tau \in (0, r_t)$ without branching, with the event that $\wh{W}_{t,\sigma_k}+\wh{h}^{\sigma_k*}_{\sigma_k} \leq -w$ for all $\sigma_k \in [\tau, r_t]$. 

Now, the probability of the first event is, up to sub-exponential terms, $\rme^{-w^2/(2\tau)} \times \rme^{-2\tau}$. This is clearly the case without the conditioning, but can be shown to hold also under the conditional measures in~\eqref{e:M1.35}. When $\wh{W}_{t,\tau} \leq -w$, an entropic repulsion effect, which is the result of conditioning the random-walk-like process plus its decorations to stay negative, makes the probability of the second event decay only polynomially in $w$ (uniformly in $t$). Multiplying the two yields $\rme^{-w^2/(2\tau)-2\tau}$ as a lower bound (on an exponential scale) on the conditional probability of $\{\cC_{t,r_t}\big([-w, 0)\big) = 0\}$ for any choice of $\tau$ and all $t$ large enough. The exponent is maximized at $\tau=w/2$, yielding a lower bound of $\rme^{-2w}$ (see Figure~\ref{f:distance}).

A matching upper bound can be obtained by stopping the process $\wh{W}_{t,s}$ at the first time $T$ when it reaches height $-w(1-\epsilon)$ for $\epsilon > 0$. Then up to this time and if $w$ is large, any branching event will result in violation of the condition $\cC\big([-w, 0)) = 0$ with probability $1-\delta$, where $\delta > 0$ can be made arbitrarily small, by choosing $\epsilon$ appropriately. This makes the probability of having no points in $[-w, 0)$ conditional on $T$ at most $\rme^{-2(1-\delta) T}$ and gives an overall upper bound (on an exponential scale) of $\rme^{-w^2/(2\tau)-2(1-\delta)\tau} \rmd \tau$ on the probability that $\cC_{t,r_t}\big([-w, 0)\big) = 0$ and $T \in \rmd \tau$, under the conditional measure in~\eqref{e:M1.35}. Integrating with respect to $\tau$, we are led to the maximization problem from before, and consequently obtain $\rme^{-(2-\delta')w}$ as an upper bound for all $t \geq 0$ and arbitrarily small $\delta' > 0$, as desired.

\subsection{Context, Extensions and Open Problems}
\label{ss:Discussion}

Branching Brownian motion is among the most fundamental random processes in modern probability theory. Aside from an intrinsic mathematical interest, the motivation for considering such a model comes from various disciplines, such as biology, where it is a canonical choice for describing population dynamics (e.g.~\cite{Fis1937}) or physics, where it can be used to model correlated energy levels in spin-glass-type systems~\cite{bovier2006tomography,DeS1988,Der1985}. In mathematics, it has deep connections with analysis, e.g. via the F-KPP equation (used by McKean~\cite{McKean} to derive asymptotics for the centered maximum) as well as other fields in probability such as random matrices~\cite{FyoKea13}, super-processes~\cite{Dawson1993}, multiplicative chaos~\cite{rhodes2013gaussian} and more. We invite the reader to consult~\cite{bovier2016gaussian,shiZhan} for recent sources on this and related models.

From the point of view of extreme value theory, results of the past few years have shown that branching Brownian motion belongs to the same universality class as other models, where correlations are  ``scale-free'' (either logarithmic or tree-like). These include the branching random walk~\cite{Aid13,Mad2015}, the two-dimensional Gaussian free field~\cite{biskup2015full,BDingZ} (and logarithmically correlated Gaussian fields in general~\cite{ding2015convergence}), characteristic polynomials of GUE ensembles~\cite{FyoSim16} and more. In all of these models the asymptotic form of the extremal process (or at least the derived law of the centered maximum) is that of a randomly shifted clustered Poisson point process with an exponential intensity, as in~\eqref{e:5}, albeit with different laws for the shift and cluster decorations. 

Statistics of extreme values of such systems are interesting for multiple reasons. From a pure-mathematical perspective, logarithmic or tree-like correlations can be thought of as the next natural step after the i.i.d. case, where the theory of extreme values is fully developed. More applicatively, the very large (or very small) values in a system often correspond to quantities of interest in the reality which the model describes. For instance, interpreting the heights as energy levels in a spin-glass system, (negative) extreme values capture the lowest energy states. The latter carry the corresponding Gibbs distribution at low temperature (glassy-phase)~\cite{carpentier2001glass,hartung2015,Madaule2015}.

Getting back to our results, the extension to branching Brownian motion with a general offspring distribution requires only minor changes in the proofs. For simplicity, we treated the binary splitting case only. All theorems and propositions will therefore still hold, albeit with different constants. Moreover, we conjecture that Theorem~\ref{thm.asyden} (with a different rate in the exponential), Theorem~\ref{t:14} and Corollary~\ref{c:1} also hold in other models, where correlations are scale-free. 

In particular, we believe that our method of proof could be applied in the case of the branching random walk and the two-dimensional Gaussian free field. This is because the three main ingredients in the proofs (see Subsection~\ref{ss:ProofOutline}): convergence of the extremal process, random walk representation of the cluster distribution and uniform tails for the centered maximum, are available in these two models as well. Nevertheless, carrying out this program requires overcoming non-trivial technical challenges and would result in a welcomed contribution to the field.

On the other hand, the statement of Proposition~\ref{prop:right.tail.cluster} depends crucially on the distribution of the difference between the heights of two nearby particles (in genealogical distance) or vertices (in lattice distance). Unlike for branching Brownian motion, where this difference can be made large by a delayed branching event, costing only an exponentially decaying probability (see Sub-Subsection~\ref{ss:Distance}), the tail of this difference is Gaussian for both the branching random walk and the Gaussian free field. We conjecture that this will result in a Gaussian decay for the probability in the statement of Proposition~\ref{prop:right.tail.cluster} and consequently also for the probability in Theorem~\ref{t:2.9}. We pose this as an open problem.

\subsection*{Organization of the Paper} 
The remainder of the paper is organized as follows. Section~\ref{s:Tools} includes the necessary technical tools to be used in the proofs thereafter. These include mainly the random walk estimates discussed above as well as the spinal decomposition and uniform bounds on the tail of the centered maximum. In Section~\ref{s:Reduction} we present the reduction statements, in which events concerning a spine particle are converted to events involving the decorated random-walk-like process. This section includes also some estimates for probabilities of such events, the proof of which uses the random-walk results from Section~\ref{s:Tools}. Next comes Section~\ref{s:Moments}, in which we use the reduction statements and the random-walk estimates to compute moments of $\cE([-v, \infty))$ subject to a truncation event restricting the height of the global maximum. These in turn are used in Section~\ref{s:ClusterProperties} to derive all results concerning cluster level sets, i.e. all propositions in Subsection~\ref{ss:ClusterProperties}. Section~\ref{s:GlobalProperties} contains the proofs of all the theorems in Subsection~\ref{ss:Global}, namely all extreme level set statements. Lastly, proofs of the random walk estimates from Section~\ref{s:Tools} can be found in the supplement material~\cite{supp}.

\section{Technical Tools}
\label{s:Tools}

In this section we introduce several technical tools which will be used throughout in the proofs to follow. Subsection~\ref{ss:RWEstimates} includes estimates on the probability that a random-walk-like process, with random time steps and decorations, stays below a curve. As explained in the proof outline (Subsection~\ref{ss:ProofOutline}), such a process arises after various reduction steps, by tracing, backwards in time, a uniformly chosen particle reaching an extreme height. Because of the randomness of the underlying branching structure, the genealogy as seen from the point of view of this distinguished (spine) particle has a biased distribution. Spinal decomposition theory can then be used to account for this bias and to convert statements involving the spine particle to ones which pertain to all particles. This is the subject of Subsection~\ref{ss:Spine}. Finally Subsection~\ref{ss:Tails} includes uniform bounds on the tail probabilities of the centered maximum. 

Although the ``random-walk'' statements in Subsection~\ref{ss:RWEstimates} are standard in flavor, the particularity of the random-walk-like process to which they apply, implies that one cannot find them ``on-the-shelf'' and new proofs have to be provided. Since these are quite lengthy and technical they have been placed in the supplemental material~\cite{supp}.

\subsection{Random Walk Estimates}
\label{ss:RWEstimates}
Let $W=(W_u :\: u \geq 0)$ be a standard one dimensional Brownian motion. Given $x,y \in \bbR$  and $0 \leq s < t$, we shall denote by $\bbP_{s,x}^{t,y}$ and
$\bbP_{s,x}$ the conditional distribution $\bbP(\cdot \,|\, W_s = x \,,\,\, W_t = y)$
and $\bbP(\cdot \,|\, W_s = x)$ respectively (if $s=0$ we assume that $W_0$ was $x$ in the first place). On the same probability space, let us suppose also the existence of a collection $Y=(Y_u :\: u \geq 0)$ of independent random variables, which is also independent of $W$. These random variables, which will be referred to as ``decorations'', satisfy 
\begin{equation}
\label{e:A1}
\forall u,z \geq 0 :\ \ 
\bbP \big(|Y_u| \geq z) \leq \delta^{-1} \rme^{-\delta z} 
\end{equation}
for some $\delta > 0$.

The third collection of random variables defined on this space, comes in the form of a Poisson point process on $\bbR$:
\begin{equation}
\label{e:A2}
\cN \sim \text{PPP}(\lambda \rmd x) \,,
\end{equation}
for some $\lambda > 0$. This process is assumed to be independent of $W$ and $Y$ and we denote by $\sigma = (\sigma_k :\: k \geq 1)$ the collections of all atoms of $\cN$, enumerated in increasing order. 

We will be interested in controlling the probability that the process $W-Y$ evaluated at all points $\sigma \cap [0,t]$ stays below a curve $\gamma_t = (\gamma_{t,u} :\: u \geq 0)$, satisfying for all $0 \leq u \leq t$,
\begin{equation}
\label{e:A3}
-\delta^{-1} \leq \gamma_{t,u} \leq \delta^{-1} \big(1+ (\wedge^t(u))^{1/2-\delta}\big) 
\quad, \qquad
\wedge^t(u) := u \wedge (t-u) \,,
\end{equation}
where $\delta \in (0,1/2)$ (to avoid using too many parameters we will use one $\delta$ in multiple conditions). The first statement is an upper bound. In this case, we might as well use the bounding function as the barrier curve itself.
\begin{prop}
\label{p:A2}
Suppose that $W, Y, \cN$ are defined as above with respect to some $\lambda > 0$ and $\delta \in (0,1/2)$. Then there exists $C = C(\lambda, \delta)$ such that for all $t \geq 0$, $x,y \in \bbR$,
\begin{equation}
\label{e:2.11}
\begin{multlined}
\bbP_{0,x}^{t,y} \Big( \max_{k :\:\sigma_k \in [0,t]} \big(W_{\sigma_k} - \delta^{-1} \big(1+(\wedge^t(\sigma_k))^{1/2-\delta}\big) - Y_{\sigma_k}\big) \leq 0 \Big) \\
\leq C \frac{(x^-+1)(y^-+1)}{t} \,,
\end{multlined}
\end{equation}

Moreover, there exists $C'= C'(\lambda, \delta)$ such that
for all $t \geq 0$ and all $x,y \in \bbR$ such that $xy \leq 0$,

\begin{equation}
\label{e:2.12}
\begin{multlined}
\bbP_{0,x}^{t,y} \Big( \max_{k:\: \sigma_k \in [0,t]} \big(W_{\sigma_k} - \delta^{-1} \big(1+(\wedge^t(\sigma_k))^{1/2-\delta}\big) - Y_{\sigma_k}\big) \leq 0 \Big) \\
\leq C' \frac{\big(x^- + \rme^{-\sqrt{2\lambda}(1-\delta) x^+}\big) 
\big(y^- + \rme^{-\sqrt{2\lambda}(1-\delta) y^+}\big)}{t} 
\exp\Big(\tfrac{(y-x)^2}{2t}\Big) \,.
\end{multlined}
\end{equation}
\end{prop}

\medskip
For an asymptotic statement, we naturally need to control the limiting behavior of both the decorations and the family of curves $\gamma = (\gamma_t)_{t \geq 0}$. For the former we assume that
\begin{equation}
\label{e:A2.5}
Y_u \overset{u \to \infty} \Longrightarrow Y_\infty \,,
\end{equation}
for some random variable $Y_\infty$. For the latter, we require that for all $u \geq 0$,
\begin{equation}
\label{e:A4}
\gamma_{t,u} \overset{t \to \infty} \longrightarrow \gamma_{\infty, u}
\ , \ \ 
\gamma_{t,t-u} \overset{t \to \infty}  \longrightarrow \gamma_{\infty, -u} \,,
\end{equation}
where $\gamma_{\infty, u}, \gamma_{\infty, -u} \in \bbR_+$ (with slight abuse, we shall use the notation $\gamma_{\infty, -0}$ for the limit of $\lim_{t \to \infty} \gamma_{t,t}$). 
We then have
\begin{prop}
\label{p:A3}
Suppose that $W, Y, \cN$ and $\gamma$ are defined as above with respect to some $\lambda>0$ and $\delta \in (0,1/2)$. Then there exists non-increasing positive functions $f,g: \bbR \to (0,\infty)$ depending on $\delta, \lambda$, $\gamma$ and $Y$, such that 
\begin{equation}
\label{e:A6}
\bbP_{0,x}^{t,y} \Big( \max_{k :\: \sigma_k \in [0,t]} \big(W_{\sigma_k} - \gamma_{t,\sigma_k} - Y_{\sigma_k} \big) \leq 0 \Big)
\sim 2 \frac{f(x)g(y)}{t} 
\ \ \text{as } t \to \infty \,,
\end{equation}
uniformly in $x,y$ satisfying $x, y \leq 1/\epsilon$ and $(x^- + 1)(y^- +1) \leq t^{1-\epsilon}$, for any fixed $\epsilon > 0$. Moreover, 
\begin{equation}
\label{e:A6.1}
\lim_{x \to \infty} \frac{f(-x)}{x} = \lim_{y \to \infty} \frac{g(-y)}{y} = 1 \,.
\end{equation}
\end{prop}

\begin{rem}[Monotonicity w.r.t. boundary conditions] \label{r:monotonicity} 
Notice that if $x \leq x'$ and $y \leq y'$, then for all $t \geq 0$ we have
\begin{equation}\label{e:montonicity}
\begin{multlined}
\bbP_{0,x}^{t,y} \Big( \max_{k :\: \sigma_k \in [0,t]} \big(W_{\sigma_k} - \gamma_{t,\sigma_k} - Y_{\sigma_k} \big) \leq 0 \Big) \\
\geq \bbP_{0,x'}^{t,y'} \Big( \max_{k :\: \sigma_k \in [0,t]} \big(W_{\sigma_k} - \gamma_{t,\sigma_k} - Y_{\sigma_k} \big) \leq 0 \Big).
\end{multlined}
\end{equation} 
Indeed, one can pass from a Brownian bridge from $x$ to $y$ to a Brownian bridge from $x'$ to $y'$ replacing $W_s$ by
$W_s - \Big( \frac{s}{t} (y'-y) +(x'-x) \big(1-\frac{s}{t} \big) \Big)$
inside the probability brackets. Since the above interpolation function is positive for every $s \in [0,t]$ we can simply lower bound it by zero to obtain \eqref{e:montonicity}. In particular, it is straightforward to show that if the convergence from Proposition~\ref{p:A3} holds, then both $f$ and $g$ are non-increasing.   
\end{rem}

We also need to know that the above asymptotics are continuous (in the sense specified below) in
$Y$ and $\gamma$. To this end for each $r \geq 0$, let $Y^{(r)}$ be a collection of random variables as $Y$ above and $\gamma^{(r)}$ be a function as $\gamma$ above, satisfying~\eqref{e:A1} and~\eqref{e:A3} uniformly for all $r \geq 0$ with some $\delta \in (0,1/2)$. Suppose that~\eqref{e:A2.5} holds for $Y^{(r)}_u$ with the limit denoted by $Y_\infty^{(r)}$ and that~\eqref{e:A4} holds with the limits denoted by $\gamma^{(r)}_{\infty, u}$ and $\gamma^{(r)}_{-\infty, u}$. 
Then
\begin{prop}
\label{p:A4}
Suppose that $W, Y, \cN, \gamma$ and $Y^{(r)}, \gamma^{(r)}$ for $r \geq 0$ are defined as above with respect to some $\lambda>0$, $\delta \in (0,1/2)$. Let $f^{(r)}$, $g^{(r)}$ be the functions $f$, $g$ respectively given in Proposition~\ref{p:A3} applied to $W, Y^{(r)}, \cN$ and $\gamma^{(r)}$. Assume that
\begin{equation}
\label{e:543}
\forall u \in [0,\infty]\!:\, Y^{(r)}_u \overset{r \to \infty}{\Longrightarrow} Y_u
\quad , \qquad
\forall u \in [0,\infty)\!:\,
\gamma^{(r)}_{\infty,\pm u} \overset{r \to \infty}{\longrightarrow} \gamma_{\infty, \pm u} \,.
\end{equation}
Then for all $x \in \bbR$,
\begin{equation}
f^{(r)}(x) \overset{r \to \infty} \longrightarrow f(x) 
\ , \quad
g^{(r)}(x) \overset{r \to \infty} \longrightarrow g(x) \,,
\end{equation}
with $f,g$ given by Proposition~\ref{p:A3} applied to $W, Y, \cN$ and $\gamma$. In particular, if $Y_\infty^{(r)} = Y_\infty$ and $\gamma_{\infty, -u}^{(r)} = \gamma_{\infty, -u}$ for all $r \geq 0$ and $u \geq 0$, then $g^{(r)}(x) = g(x)$ for all $r \geq 0$.	
\end{prop}

\subsection{Spinal Decomposition}
\label{ss:Spine}
A key tool for reducing the computation of moments of the number of particles satisfying a certain condition is the so-call spinal decomposition, in the form of the two lemmas below. We refer the reader to~\cite{HarRob2011} for a more general and thorough treatment of this method, as well as an historical overview.

For integer $k \geq 1$, the $k$-spine branching Brownian motion describes particles which branch and diffuse as in the original process, only that in addition they may carry ``marks'' indexed by the set $\{1, \dots, k\}$, which affect their branching and/or diffusion laws. For our purposes, we can assume that the diffusion law is always that of a standard Brownian motion and splitting is always binary, regardless of the carried marks. What is affected by the marks, is the branching rate, which is $2^m$ if the particle carries $m$ marks. In addition, once a particle branches, each mark is transferred to one of its two children with equal probability and independently of the other marks. 

As before, the set of particles at time $t$ will be denoted by $L_t$, which again we equip with the genealogical metric $\rmd = \rmd_t$. The positions of particles will be given by the random collection $h_t = (h_t(x) :\: x \in L_t)$, again exactly as before. The new information, namely the location of the marks at time $t$, will be denoted by the collection $X_t = (X_t(l) :\: l=1, \dots, k)$, where $X_t(l) \in L_t$ is the particle holding mark $l$ at time $t$. The genealogical line of decent of particle $X_t(l)$, namely the function $t \mapsto X_t(l)$, will be referred to as the $l$-th spine of the process. 

We shall denote by $\wt{\bbP}^{(k)}$ the underlying probability measure and by $\wt{\bbE}^{(k)}$ the corresponding expectation. To simplify the notation in the case $k=1$, we shall write $\wt{\bbP}$, $\wt{\bbE}$ and $X_t$ in place of $\wt{\bbP}^{(1)}$, $\wt{\bbE}^{(1)}$ and $X_t(1)$. Note that in the case $k=0$ the process is reduced to a regular branching Brownian motion, in which case we will keep using the notation $\bbP$, $\bbE$ and use $(\cF_t :\: t \geq 0)$ to denote its natural filtration.

The first lemma shows how to reduce first moment computations for regular branching Brownian motion  to expectations involving the $1$-spine measure. To avoid integrability issues, we state it for a bounded function, although this is entirely not necessary.
\begin{lem}[Many-to-one]
\label{l:4.1}
Let $F = (F(x) :\: x \in L_t)$ be a bounded $\cF_t$-measurable real-valued random function on $L_t$. Then, 
\begin{equation}
\bbE \Big( \sum_{x \in L_t} F(x) \Big)
= \rme^{t} \,\wt{\bbE} F(X_t) \,.
\end{equation}
\end{lem}

The second lemma is suitable for second moment computations.
\begin{lem}[Many-to-two]
\label{l:4.2}
Let $F = (F(x,y) :\: x,y \in L_t)$ be a bounded $\cF_t$-measurable real-valued random function on $L_t \times L_t$. Then, 
\begin{equation}
\bbE \Big( \sum_{x,y \in L_t} F(x,y) \Big)
= \rme^{3t}\, \wt{\bbE}^{(2)} \Bigl(F\big(X_t(1), X_t(2)\big) \, \rme^{-\rmd(X_t(1), X_t(2))} \Bigr)\,.
\end{equation}
\end{lem}

\begin{rem}
\label{r:4.5}
Observe that on the event $\{\rmd(X_t(1), X_t(2)) = r\}$ for some $0 \leq r \leq t$, at all branching events prior to time $t-r$, which occur at rate $4$, both spine particles ``chose'' to follow the same child. Since such events have probability $1/2$ and they are independent of each other, standard Poisson thinning arguments show that conditional on $\{\rmd (X_t(1), X_t(2)) = r\}$ branching along the line of descent of the two spine particles up to time $t-r$ occurs at rate $2$. Since the motion is not effected by the conditioning, we see that under the conditioning, the two-spine process behaves as a one-spine process up to time $t-r$, with the two spine particles identified. The same reasoning also implies that $\big(t-\rmd(X_t(1), X_t(2)) \big) \overset{law} = e \wedge t$, where $e$ is an exponential random variable with rate $2$.
\end{rem}

\subsection{Uniform Tail Estimates for the Centered Maximum}
\label{ss:Tails}
Even though asymptotics for the upper tail are well known, precise asymptotics for the lower tail are harder to find. Recall that we are writing $h_t^*$ for $\max_{x 
\in L_t} h_t(x)$.
\begin{lem}
\label{l:TailBBMMax}
There exists $C, C' > 0$, such that for all $t \geq 0$ and $u \geq 0$,
\begin{equation}
\bbP(h_t^* - m_t > u) \leq C u \rme^{- \sqrt{2} u}
\quad \text{and} \quad 
\bbP(h_t^* - m_t < -u) \leq  C' \rme^{-(2 - \sqrt{2}) u} \,.
\end{equation}
\end{lem}

\begin{proof}

A sharper bound for the right-tail probabilities was obtained in Corollary~10 of~\cite{ABK_G}. 
For the left tail, we can appeal to both~\cite{B_C} and~\cite{ABK_G}. From the first reference, we now that $u(t,x) := \bbP(h_t^* > x)$ is the unique solution to the F-KPP equation with heavy-side initial data, and that for any $x < 0$, the function $u(t,\wt{m}_t + x)$, where $\wt{m}_t$ is the median of $u(t,\cdot)$, is decreasing in $t$ and converges to $\omega(x)$, with $\omega$ forming the so-called traveling wave solution of the F-KPP equation. Moreover, it is shown in~\cite{B_C} that $|m_t - \wt{m}_t|$ stays bounded uniformly in $t \geq 0$. On the other hand, in~\cite{ABK_G} (Appendix~A of the [v1]~arXiv version), the authors show that $1-\omega(-x) \sim \rme^{-(2 - \sqrt{2})x}$ as $x \to \infty$. Combing the above, the bound on the left-tail follows.
\end{proof}

\section{Reduction to a Decorated Random-Walk-Like Process}
\label{s:Reduction}
In the sequel we shall need to estimate probabilities concerning the height of one or two spine particles and the clusters around them, subject to a restriction on the global maximum of the process. By tracing the spine particles backwards in time, such events can be recast in terms of a decorated random-walk-like process, for which asymptotic probabilities are given in Subsection~\ref{ss:RWEstimates}. We therefore proceed by defining this process explicitly and then stating various reduction lemmas which will be needed in the sequel. The section concludes with a few lemmas in which the probability of events involving the decorated process are estimated.  These estimates will be used frequently in the proof to follow.

\subsection{Definition of the Walk and Reduction Statements}
As before let $W=(W_s :\: s \geq 0)$ be a standard Brownian motion, whose initial position we leave free to be determined according to the conditional statements we make. For $0 \leq s \leq t$, we fix
\begin{equation}
\label{e:20.5}
\gamma_{t,s} := \tfrac{3}{2\sqrt{2}} \big(\log^+ s - \tfrac{s}{t}\log^+ t \big) \quad \text{and}  \qquad
\wh{W}_{t,s} := W_s - \gamma_{t,s} \,.
\end{equation}
We shall also need the collection $H=\big(h^s = (h^s_t)_{t \geq 0} :\: s \geq 0\big)$ of independent copies of $h$, that we will assume to be independent of $W$ as well. Finally, let $\cN$ be a Poisson point process with intensity $2 \rmd x$ on $\bbR_+$, independent of $H$ and $W$ and denote by $\sigma_1 < \sigma_2 < \dots$ its ordered atoms. The triplet $(\wh{W}, \cN, H)$ forms the decorated random-walk-like process, which was eluded to in the beginning.

To see the relevance of the above process, recall that $\rmB_{r}(x)$ is the ball of radius $r$ around $x$ in the genealogical distance $\rmd$, and that we write $\wh{h}_t = h_t -m_t$ and $\wh{h}^*_t = \max_{x \in L_t} \wh{h}_t(x)$. For $A \subseteq L_t$ set also $\wh{h}_t^*(A)$ for $\max_{x \in A} \wh{h}_t(x)$, then,
\begin{lem}
\label{l:5.1}
For all $0 \leq r \leq t$ and $u, w \in \bbR$,
\begin{equation}
\label{e:29.2}
\begin{split}
\wt{\bbP} & \Big( \wh{h}_t^*\big(\rmB^\rmc_{r}(X_t)\big)
 \leq u \, \Big| \, \wh{h}_t(X_t) = w \Big) \\
& = \bbP \Big( \max_{k: \sigma_k \in [r,t]} \big(\wh{W}_{t,\sigma_k} + \wh{h}^{\sigma_k*}_{\sigma_k}\big) \leq 0 \ \Big|\ 
	\wh{W}_{t,r} = w - u,\, \wh{W}_{t,t} = -u \Big) \,.
\end{split}
\end{equation}
In particular for all $t \geq 0$ and $v,w \in \bbR$,
\begin{equation}
\label{e:29.1}
\begin{multlined}
\wt{\bbP} \big( \wh{h}^*_t \leq u \, \big| \, \wh{h}_t(X_t) = w \big) \\
= \bbP \Big( \max_{k: \sigma_k \in [0,t]} \big(\wh{W}_{t,\sigma_k} + \wh{h}^{\sigma_k*}_{\sigma_k}\big) \leq 0 \ \Big|\ 
	\wh{W}_{t,0} = w - u,\, \wh{W}_{t,t} = -u \Big) \,.
\end{multlined}
\end{equation}
\end{lem}
\begin{proof}
Since both Brownian motion and Poison point process are distributional invariant under time reversal, tracing the spine particle backwards in time, the left hand side of~\eqref{e:29.2} can be written as
\begin{equation}\label{li.R1}
\bbP \Big(\max_{k :\: \sigma_k \in [r, t]} \big(W_{\sigma_k} + h^{\sigma_k*}_{\sigma_k}\big) \leq m_t + u
\,\Big|\, W_0 = m_t + w \,,\, W_t = 0 \Big) \,.
\end{equation}
where $W_s$, $\sigma_k$ and $h^\sigma_t$ are as above.

Now independence of $\cN$, $W$ and $H$ together with standard Gaussian properties enjoyed by $W$ imply that the probability above does not change if we replace $W_s$ by $W_s +u+ m_t (t-s)/t $ everywhere in \eqref{li.R1}. 
Replacing $h^{s}_{s}$ and $W_s$ by $\wh{h}^{s}_{s} + m_s$ and by $\wh{W}_{t,s} + \gamma_{t,s}$ respectively and observing that  $m_t s/t - m_s = \gamma_{t,s}$, we obtain~\eqref{e:29.2}, then~\eqref{e:29.1} follows by plugging in $r=0$.
\end{proof}

In a similar way, we can express the distribution of the cluster around the spine particle, given that it reaches height $m_t$. For what follows $\cE_t^s$ denotes the extremal process of $h^s_t$, defined as in~\eqref{e:0.1} only with respect to $h^s_t$ in place of $h_t$.

\begin{lem}
\label{l:5.2}
Let $\cA_t := \Big \{
\max\limits_{k : \, \sigma_k \in [0,t]} \big(\wh{W}_{t,{\sigma_k}} + \wh{h}^{\sigma_k*}_{\sigma_k}\big) \leq 0 \Big\}$, then for all $0 \leq r \leq t$ we have that 
\begin{equation}
\label{e:26.6}
\begin{split}
& \wt{\bbP}  \Big( \big(\cC_{t,r} (X_t)\,,\, (h_{t-s}(X_{t-s}) - m_t)_{s \leq r} \big) 
 \in \cdot \,\Big|\, \wh{h}^*_t = \wh{h}_t(X_t) = 0 \Big)  \\
&  = \! \textstyle \bbP \bigg( \! 
\Big( \! \sum\limits_{\sigma_k \leq r } \! \cE_{\sigma_k}^{\sigma_k} \big(\cdot - \wh{W}_{t,\sigma_k}\big),\, 
(\wh{W}_{t,s} \!- \! m_s)_{s \leq r }\! \Big)\! \in \! \cdot 
	 \ \Big| \ \wh{W}_{t,0} \! = \! \wh{W}_{t,t} \! = \! 0 \,;\, \cA_t \, \bigg).
\end{split}
\end{equation}
\end{lem}
\begin{proof}
As in the proof of Lemma \ref{l:5.1}, we can replace $h_{t-s}(X_{t-s})$ in the left hand side of \eqref{e:26.6} by $W_s$, so that the conditioning event in the left hand side of \eqref{e:26.6} reads  
\begin{equation}
\textstyle
\{ W_0 = m_t \,,\,\, W_t = 0 \,,\,\, \max_{k:\: \sigma_k \in [0,t]}  W_{\sigma_k} + h^{\sigma_k*}_{\sigma_k} - m_t \leq 0 \},
\end{equation}
and $\cC_{t,r} (X_t) \stackrel{\mathrm{law}}{=} \sum_{x \in L^{\sigma_k}_{\sigma_k}} \nolimits \delta_{W_{\sigma_k} + h^{\sigma_k}_{\sigma_k}(x) - m_t}$.
The result follows after applying the same transformations as in the proof of Lemma \ref{l:5.1}.
\end{proof}

The advantage of the above formulation, which uses the decorated random walk $\wh{W}_{t,s}$, is that it is suitable for an application of the random walk estimates from Subsection~\ref{ss:RWEstimates}, provided that $\gamma_{t,s}$ from~\eqref{e:20.5} and $\wh{h}^{s*}_{s}$ satisfy the required conditions. Lemma~\ref{l:TailBBMMax} shows that $\wh{h}^{s*}_s$ satisfies the tail conditions with $\delta < (2 - \sqrt{2})^{-1}$. To check the conditions for $\gamma_{t,s}$, we shall need the following technical lemma, whose proof is elementary.

\begin{lem} 
\label{l:lisa}
Let $s,r, t\in \bbR$ be such that $0 \leq r \leq r+s \leq t$, then
\begin{equation}\label{e:aser1}
 - 1  \, \leq \log^+ (r+s) - \left(\tfrac{t-(r+s)}{t-r} \log^+r + \tfrac{s}{t-r} \log^+t\right)
\leq \, 1+ \log^+ \big(s \wedge (t-r-s) \big).
\end{equation}
\end{lem}
\begin{proof}
Starting with the lower bound, it follows from the concavity of $\log$ that $0$ is lower bound when $r \geq 1$. If $r < 1$ and $r+s < 1$, then the middle expression is equal to $ -s(\log^+ t)/ (t-r)$, which is again grater than $-1$.
Lastly if $r < 1$ but $r+s \geq  1$, then the middle expression is equal to $\log (r +s )-s(\log t)/ (t-r)$, whose minimum, attained at $s = 1-r $, is again greater than $-1$.

For the upper-bound, we consider the two cases $s \leq (t-r)/2 $ and $s > (t-r)/2$ separately. In the first case, by replacing $\log^+t$ by $\log^+r$ in the middle expression, it is enough to prove the upper bound for $\log^+(r+s)  - \log^+r$. But, concavity of $\log$ implies that the latter is at most $1 + \log^+s$, which proves the statement for $s \leq (t-r)/2 $. 
On the other hand, if $s > (t-r)/2$ we set $s' = t-r-s$ and rewrite the middle expression in~\eqref{e:aser1} as
\begin{equation}
\log^+ (t-s') - \left(\tfrac{s'}{t-r} \log^+r + \tfrac{t-r-s'}{t-r} \log^+t\right) 
\leq \big( \log^+ t - \log^+ r \big) \tfrac{s'}{t-r} \,.
\end{equation}
Above, to get the second inequality, we have bounded $\log^+(t-s')$ by $\log^+t$.
Appealing to concavity of the logarithm function again, if $r \geq 1$, then the right hand side above is further upper bounded by $\log (s'+r)-\log r$ which is again smaller than $1 + \log^+ s'$ as before, which is what we need to show in this case. If $r < 1$ and $t < 1$, then the upper bound is trivial. Finally, if $r < 1$ and $t \geq 1$, then the upper bound follows from the inequality    
$s' \log t \leq  (t-1) \big(1+ \log^{+} s' \big)$ which holds for all $s'\leq t$.
\end{proof}

\subsection{Fundamental Estimates}
With the above result at hand, we can state the following two lemmas, which are essentially corollaries of the random walk estimates from Subsection~\ref{ss:RWEstimates}. In the first one, we obtain upper bounds and asymptotics for the probabilities appearing in Lemma~\ref{l:5.1}.
\begin{lem}
\label{lem:15}
There exists $C, C' > 0$ such that for all $0 \leq r \leq t$ and $w, v \in \bbR$,
\begin{equation}
\label{e:52}
\begin{multlined}
\bbP \Big( \max_{k: \sigma_k \in [r,t]} \big(\wh{W}_{t,\sigma_k} + \wh{h}^{\sigma_k*}_{\sigma_k}\big) \leq 0 
	\,\Big|\, \wh{W}_{t,r} = v \,,\,\, \wh{W}_{t,t} = w \Big)  \\
	\leq C \frac{(v^- + 1)(w^- + 1)}{t-r} \,.
\end{multlined}
\end{equation}
and if $vw \leq 0$ then,
\begin{equation}
\label{e:52.1}
\begin{multlined}
\bbP \Big( \max_{k: \sigma_k \in [r,t]} \big(\wh{W}_{t,\sigma_k} + \wh{h}^{\sigma_k*}_{\sigma_k}\big) \leq 0 
	\,\Big|\, \wh{W}_{t,r} = v \,,\,\, \wh{W}_{t,t} = w \Big)  \\
	\leq C' \frac{\big(v^- + \rme^{-\frac{3}{2}v^+}\big)\big(w^- + \rme^{-\frac{3}{2}w^+}\big)}{t-r} 
	\exp \Big(\tfrac{(v-w)^2}{2(t-r)} \Big)\,.
\end{multlined}
\end{equation}
Also, there exists non-increasing functions $g: \bbR \to (0,\infty)$ and $f^{(r)}: \bbR \to (0,\infty)$ for $r \geq 0$, such that for all such $r$
\begin{equation}
\label{e:53}
\bbP \Big( \max_{k: \sigma_k \in [r,t]} \big(\wh{W}_{t,\sigma_k} + \wh{h}^{\sigma_k*}_{\sigma_k}\big) \leq 0 
	\,\Big|\, \wh{W}_{t,r} = v \,,\,\, \wh{W}_{t,t} = w \Big) 
	\sim 2 \frac{f^{(r)}(v) g(w)}{t-r},
\end{equation}
as $t \to \infty$ uniformly in $v,w$ satisfying $v,w < 1/\epsilon$ and $(v^-+1)(w^-+1) \leq t^{1-\epsilon}$ for any fixed $\epsilon > 0$. Moreover, 
\begin{equation}
\label{e:53.5}
\lim_{v \to \infty} \frac{f^{(r)}(-v)}{v} = \lim_{w \to \infty} \frac{g(-w)}{w} = 1  \,,
\end{equation}
for any $r \geq 0$. Finally there exists $f: \bbR \to (0, \infty)$ such that for all $v \in \bbR$, 
\begin{equation}
\label{e:54}
f^{(r)}(v) \overset{r \to \infty} \longrightarrow f(v) \,.
\end{equation}
\end{lem}
\begin{proof}
Given $r,t,v,w$ satisfying the above assumptions, let $t^{(r)} := t-r$. By tilting and shifting we can replace $\wh{W}_{t,s}$ everywhere inside the probability on the left hand side of~\eqref{e:52} by
$\wh{W}_{t,s} + \gamma_{t,r} + \frac{s-r}{t^{(r)}} \big( \gamma_{t,t} - \gamma_{t,r} \big)
= W_s - \gamma_{t,s} + \gamma_{t,r} + \frac{s-r}{t^{(r)}} \big( \gamma_{t,t} - \gamma_{t,r} \big)$. 
Setting
\begin{equation}
\label{e:60B}
\gamma^{(r)}_{t^{(r)}, s} := \gamma_{t,s+r} - \gamma_{t,r} - \frac{s}{t^{(r)}} \big( \gamma_{t,t} - \gamma_{t,r} \big) 
\, , \qquad 
Y^{(r)}_{s} := -\wh{h}^{(s+r)*}_{s+r},
\end{equation}
and using shift law invariance of $W$ and $\cN$, the left hand side of~\eqref{e:52} now reads
\begin{equation}
\bbP \Big( \max_{k: \sigma_k \in [0,t^{(r)}]} \big(W_{\sigma_k} - \gamma^{(r)}_{t^{(r)}, \sigma_k} - Y^{(r)}_{\sigma_k} \big) \leq 0 
	\,\Big|\, W_0 = v \,,\,\, W_{t^{(r)}} = w \Big).
\end{equation}
 
Next, we want to apply Propositions~\ref{p:A2},~\ref{p:A3} and~\ref{p:A4}. We just need to make sure that the conditions required by these propositions hold. By assumption,~\eqref{e:A2} holds with $\lambda = 2$ and thanks to Lemma~\ref{l:TailBBMMax} we know that~\eqref{e:A1} holds with any $\delta$ small enough uniformly in $r$. Finally, using Lemma~\ref{l:lisa} noting that the middle expression in~\eqref{e:aser1} is exactly $(2\sqrt{2}/3) \gamma^{(r)}_{t^{(r)},s} $, we have
\begin{equation}
\label{e:58B}
-1 \leq (2\sqrt{2}/3) \gamma^{(r)}_{t^{(r)},s} \leq 1 + \log^+ \wedge^{t^{(r)}}(s) \: : \ 0 \leq r \leq t, 
\ \text{ and } \  0 \leq s \leq t^{(r)} \,,
\end{equation}
which shows that Condition~\eqref{e:A3} holds with any $\delta <1/2$.  This implies that for any $r \geq 0$ both statements in Proportion~\ref{p:A2} apply, provided that we choose $\delta \in (0,1/2)$ small enough. In particular, by decreasing $\delta$ if necessary, we may and will assume that $\sqrt{2 \lambda} (1- \delta) = 2 (1- \delta) \geq 3/2$, which yields~\eqref{e:52} and~\eqref{e:52.1}.

Turning now to~\eqref{e:53},~\eqref{e:53.5} and~\eqref{e:54}, a bit of algebra shows that for fixed $r$ and all $s \geq 0$
\begin{equation}
\begin{split}
\lim_{t^{(r)}  \to \infty } & \gamma^{(r)}_{t^{(r)}, s} = \tfrac{3}{2\sqrt2}\big(\log^+ (s + r) - \log^+ r\big) 
=: \gamma^{(r)}_{\infty, s}, \\ 
\lim_{t^{(r)}  \to \infty } & \gamma^{(r)}_{t^{(r)}, t^{(r)}-s} = 0 
=: \gamma^{(r)}_{\infty, -s} \,,
\end{split}
\end{equation}
 while the convergence of the centered maximum gives,
$Y^{(r)}_{s} \Longrightarrow Y$ as $s \to \infty$ or $r \to \infty$, where $Y$ has the limiting law of the centered maximum. Moreover, for all $s \geq 0$ clearly $\gamma^{(r)}_{\infty, s} \longrightarrow 0$ as $r \to \infty$. Therefore the conditions of Proposition~\ref{p:A3} and Proposition~\ref{p:A4} are satisfied implying~\eqref{e:53},~\eqref{e:53.5} and~\eqref{e:54}.
\end{proof}

\section{Truncated Moments of the Level Set Size}
\label{s:Moments}
The goal in this section is to estimate the first and second moments of the number of particles lying above $m_t + v$ for $v \in \bbR$. Since the expectation of such quantities blows up as $t \to \infty$, one has to introduce a truncation event. Unlike the usual truncation event (introduced by Bramson in~\cite{Bramson1978maximal}), 
whereby the trajectory of such particle is constrained to lie below a curve, we choose to use the event that the global maximum stays below a certain value, namely $\{h^*_t \leq m_t + u\}$. This truncation can be more conveniently used later, when we derive cluster properties (Section~\ref{s:ClusterProperties}). In light of the tightness of the centered global maximum, the probability of this event tends to $0$ when $u \to \infty$ uniformly in $t$. Therefore, for the sake of distributional results, we can always work under this restriction and remove it just in the very end.

Recall the definition of the extremal process from~\eqref{e:0.1}. Since for every Borel set $A \subseteq \bbR$
\begin{equation}
\label{e:58.2}
\begin{split}
\cE_t(A) & 1_{\{\wh{h}^*_t \leq u\}}
 =\sum_{x \in L_t} 1_{\{\wh{h}_t(x) \in A \,,\,\, \wh{h}_t^* \leq u\}} \\
& \quad \text{and} \quad
\cE_t(A)^2 1_{\{\wh{h}^*_t \leq u\}}
=\sum_{x,y \in L_t} 1_{\{\wh{h}_t(x) \in A \,,\,\, \wh{h}_t(y) \in A \,,\,\,
\wh{h}_t^* \leq u\}}\,.
\end{split}
\end{equation}
we can use the spinal decomposition in the form of the many-to-one and many-to-two lemmas in Subsection~\ref{ss:Spine}, to compute the expectation of the quantities above, provided we can estimate the probabilities, under the corresponding spine measures, of the events in the sums, with $x,y$ replaced by the spine particles $X_t(1)$, $X_t(2)$, respectively. We start with the first moment.

\subsection{First Moment}
Recall that the one-spine measure as introduced in Subsection~\ref{ss:Spine} is denoted by $\wt{\bbP}$ and the corresponding expectation is $\wt{\bbE}$.
\begin{lem}
\label{l:16}
There exists $C, C' > 0$ such that for all $t \geq 0$ and $v \leq u$,
\begin{equation}
\label{e:64}
\wt{\bbP} \big(\wh{h}_t(X_t) \geq v \,,\,\, \wh{h}^*_t \leq u \big)
\leq C \rme^{-t} \rme^{-\sqrt{2} v} \big(u-v+1 \big) \big(u^+ + 1 \big) \big(
\rme^{-\frac{v^2}{4t}} + \rme^{\frac{v}{2}} \big) \,,
\end{equation}
in addition, if $u \leq 0$ then we also have that
\begin{equation}
\label{e:64.3}
\wt{\bbP} \big(\wh{h}_t(X_t) \geq v \,,\,\, \wh{h}^*_t \leq u \big)
\leq C' \rme^{-t} \rme^{-\sqrt{2} v} \big(u-v+1 \big) \rme^{-\frac{3}{2}u^-}.
\end{equation}
Moreover with $g: \bbR \to (0,\infty)$ from Lemma~\ref{lem:15} we have that uniformly in $u,v$ satisfying $|u| \leq 1/\epsilon$ and $t^{\epsilon} < u-v < t^{1-\epsilon}$, for any fixed $\epsilon > 0$,
\begin{equation}
\label{e:64.1}
\wt{\bbP} \big(\wh{h}_t(X_t) \geq v \,,\,\, \wh{h}^*_t \leq u \big)
\sim \rme^{-t} \rme^{-\sqrt{2} v - \frac{v^2}{2t}} (u-v+1) \frac{g(-u)}{\sqrt{\pi}} 
\quad \text{as $t \to \infty$}. 
\end{equation}

\end{lem}
\begin{proof}

Starting with the first upper bound,  we write the left hand side of~\eqref{e:64} as the integral
\begin{equation}
\label{e:63.1}
\int_{w=v}^u \wt{\bbP} \big( \wh{h}^*_t \leq u \,\big|\, \wh{h}_t(X_t) = w \big) 
	\wt{\bbP} \big( \wh{h}_t(X_t) \in \rmd w \big) \,.
\end{equation} 
Using the second part of Lemma~\ref{l:5.1} and then the first upper bound in Lemma~\ref{lem:15}, the conditional probability in the integral is bounded above by $C t^{-1}  (u-w + 1)(u^+ + 1)$.
At the same time, $\wh{h}_t(X_t)$ is Gaussian with mean $-m_t := -\sqrt{2}t + \frac{3}{2\sqrt{2}} \log^+ t$ and variance $t$. Therefore,
\begin{equation}
\label{e:72}
\begin{split}
\frac{\wh{\bbP} \big( \wh{h}_t(X_t) \in \rmd w \big)}{\rmd w} & =  \frac{t \rme^{-t}\rme^{-\sqrt{2}w}}{\sqrt{2\pi}}  \exp \bigg(-\tfrac{ \big(w- \frac{3}{2\sqrt{2}} \log^+ t \big)^2}{2t} \bigg)  \\
& \leq C t \rme^{-t} \exp \Big( -\sqrt{2} w - \tfrac{w^2}{4t} \Big) \,
\end{split}
\end{equation}
Using these inequalities in~\eqref{e:63.1} we may bound the integral by
\begin{equation}
\label{li.T1}
C \rme^{-t}(u^+ + 1) (u-v+1) \times
\begin{cases}
	\rme^{-\frac{v^2}{4t}-\sqrt{2}v} 	  &: v \geq (-\sqrt{8} + \eta)  t \,,\\
	\rme^{(2+\eta)t} &: v < (-\sqrt{8} + \eta) t \,,
\end{cases} 
\end{equation}
with any $\eta > 0$ and $C = C(\eta) > 0$. Choosing $\eta$ small enough, the last factor in \eqref{li.T1} can be bounded by $\rme^{-\sqrt{2}v}\big(\rme^{-v^2/4t} + \rme^{v/2}\big)$, which gives the upper bound.

Now, if $u \leq 0$, then $v \leq w \leq 0$ and consequently the left hand side of~\eqref{e:72} can be bounded by 
$C t \rme^{-t} \times \exp\big(-\sqrt{2} w - \frac{w^2}{2t}\big)$. 
Observing that $w-u \leq 0$, we 
now use the second upper bound in Lemma~\ref{lem:15} to estimate the first term in the integral in~\eqref{e:63.1}. The probability in question is now bounded by
\begin{equation}
C \rme^{-t}  \rme^{-\frac{3}{2}u^-} \int_{w=v}^u \rme^{-\sqrt{2}w} (u-w+1) \rmd w  \,,
\end{equation}
which is smaller than the right hand side of~\eqref{e:64.3} for a proper constant $C' > 0$.

As for the asymptotic statement, we use Lemma~\ref{l:5.1} again and then Lemma~\ref{lem:15} with $r=0$ for the first term in the integral, but this time we use~\eqref{e:53} in order to obtain asymptotics. This gives 
\begin{equation}
\wt{\bbP} \big( \wh{h}^*_t \leq u \,\big|\, \wh{h}_t(X_t) = w \big) \sim 2 \, \frac{f^{(0)}(w-u) g(-u)}{t} \,,
\end{equation}
as $t \to \infty$, uniformly in $u,v$ as specified in the statement and any $w \in [v,u]$. Using~\eqref{e:72} we also have that uniformly in $w \in [v,u]$
\begin{equation}
\frac{\wh{\bbP} \big( \wh{h}_t(X_t) \in \rmd w \big) }{\rmd w} 
\sim \frac{t \rme^{-t}}{\sqrt{2 \pi}} \exp \Big(-\sqrt{2} w- \tfrac{w^2}{2t} \Big)
\qquad \text{as $t \to \infty$}.
\end{equation}
Plugging these estimates in~\eqref{e:63.1}, the integral there is uniformly asymptotic to $ (2/ \pi\big)^{1/2} \rme^{-t} g(-u)$ times
\begin{equation}
\label{e:572}
\begin{split}
&  \int_{w=v}^u  f^{(0)}(w-u)\exp \Big(-\sqrt{2} w- \tfrac{w^2}{2t} \Big) \,  \rmd w \\
& \ =  (u-v+1) \rme^{-\sqrt{2} v - \frac{v^2}{2t}} 
\int_{y=0}^{u-v} \, \frac{f^{(0)}(v-u+y)}{u-v+1} \rme^{-\sqrt{2} y- \tfrac{y^2}{2t} - \tfrac{yv}{t} } \, \rmd y ,
\end{split}
\end{equation}
where we have also substituted $y=w-v$ to obtain the second line above.

Since $f^{(0)}(-x) \sim x$ as $x \to \infty$ and $u-v \geq t^\epsilon$ the ratio in the integrand is bounded by above and tends to $1$ as $t \to \infty$, with convergence uniform in $y = o (t^\epsilon)$. Moreover, since $|v|t^{-1} = o(1)$ uniformly as $t \to \infty$, the above integral restricted to $y > \log t$ vanishes as $t \to \infty$. 
On the other hand, when $y \in [0,\log t]$ the integrand converges uniformly to $\rme^{-\sqrt{2}y}$ as $t \to \infty$, implying that the integral itself converges uniformly to $\int_0^{\infty} \rme^{-\sqrt{2} y} \rmd y =1/ \sqrt{2}$, which yields~\eqref{e:64.1}. 
\end{proof}

We are now in a position to estimate the first moment of $\cE_t \big([v, \infty) \big)$ under the restriction that~$\wh{h}^*_t \leq u$.
\begin{lem}
\label{l:18}
There exists $C > 0$ such that for all $t \geq 0$ and $v \leq u$, 
\begin{equation}
\label{e:2001}
\bbE \Big(\cE_t\big([v, \infty)\big) ;\; \wh{h}^*_t \leq u \Big)
\leq C \rme^{-\sqrt{2} v} (u-v+1) (u^+ + 1) \big(
\rme^{-v^2/4t} + \rme^{v/2} \big) \,. 
\end{equation}
Moreover with $g: \bbR \to (0,\infty)$ from Lemma~\ref{lem:15} we have that uniformly in $u,v$ satisfying $|u| \leq 1/\epsilon$ and $t^{\epsilon} < u-v < t^{1-\epsilon}$, for any fixed $\epsilon > 0$,
\begin{equation}
\label{e:2002}
\bbE \Big(\cE_t\big([v, \infty)\big) ;\; \wh{h}^*_t \leq u \Big)
\sim \rme^{-\sqrt{2} v - \frac{v^2}{2t}} (u-v+1) \frac{g(-u)}{\sqrt{\pi}} 
\qquad \text{as $t \to \infty$.}
\end{equation}
\end{lem}
\begin{proof}
Writing $\cE_t\big([v, \infty) \big)1_{\{\wh{h}^*_t \leq u\}}$ as $\sum_{x \in L_t} F(x)$ with $F(x)$ being the indicator function $ 1_{\{\wh{h}_t(x) \geq v,\, \wh{h}^*_t \leq u\}}$, we may apply the (many-to-one) Lemma~\ref{l:4.1} and then use Lemma~\ref{l:16} to estimate the resulting integral, the result follows.
\end{proof}

\subsection{Second Moment}
For the second moment we only need an upper bound. 
Recall that the two-spine measure as introduced in Subsection~\ref{ss:Spine} is denoted by $\wt{\bbP}^{(2)}$ and the corresponding expectation is $\wt{\bbE}^{(2)}$.
\begin{lem}
\label{lem:18.5}
There exists $C > 0$ such that for all $0 \leq r \leq t$ and $v \leq u$,
\begin{equation}
\label{e:699}
\begin{split}
&\wt{\bbP}^{(2)}  \Big( \! \min \big\{ \wh{h}_t(X_t(1)) , \wh{h}_t(X_t(2)) \big\} \! \geq v,\,\,
		   \wh{h}_t^* \leq u \,\big|\, \rmd(X_t(1), X_t(2)) = r \Big) \\ 
& \ \ \leq C \frac{\rme^{-t-r} \Big[\rme^{\sqrt{2}u}  (u^+ + 1) 
	\rme^{-2\sqrt{2} v} (u-v+1)^2\Big]}{1+\big(r \wedge (t-r) \big)^{3/2}}  \Big(\rme^{-\frac{(u-v)^2}{4t}} \! + \! \rme^{-\frac{(u-v)}{2}}\Big) .
\end{split}
\end{equation}
\end{lem}
\begin{proof}
In light of Remark~\ref{r:4.5}, by conditioning further on the position of $h_{t-r} \big(X_{t-r}(1)\big)$ (which is also the position of $h_{t-r} \big(X_{t-r}(2)\big)$) the left hand side of \eqref{e:699} can be written as
\begin{equation}
\label{e:502}
\begin{split}
\int_z \wt{\bbP} & \big( \wh{h}_r(X_r) \geq v-z \,, \wh{h}_r^* \leq u-z \big)^2  \\ 
 & \times \wt{\bbP}\Big( \wh{h}_{t-r}(X_{t-r}) - m_{t,r} \in \rmd z, \,\wh{h}^*_t \big(\rmB_{r}(X_t)^\rmc \big) \leq u \Big) ,
\end{split}
\end{equation}
where $m_{t,r} := m_t - m_r - m_{t-r} = \frac{3}{2\sqrt{2}} \big(\log^+ r + \log^+ (t-r) - \log^+ t \big)$ and $X_t$ is the one-spine particle. Observe that $m_{t,r}$ is always non-negative and satisfies
\begin{equation}
\label{e:574}
m_{t,r} - \tfrac{3}{2\sqrt{2}} \log^+ \wedge^t(r)  \in [-2\log 2, 0] \,.
\end{equation}

To bound the second term in the integrand, we use Lemma~\ref{l:5.1} to express it as $\wt{\bbP}  \big( \, \wh{h}_{t-r}(X_{t-r}) - m_{t,r} \!  \in \rmd z \big)$ times
\begin{equation} \label{e:as1}
\bbP  \Big( \max\limits_{k: \sigma_k \in [r,t]} \big(\wh{W}_{t,\sigma_k} \!  + \wh{h}^{\sigma_k*}_{\sigma_k}\big) \!  \leq 0 
\ \Big| \  \wh{W}_{t,r} = z  +  m_{t,r} \! - u,\, \wh{W}_{t,t} =  -u \Big) \,.
\end{equation}
Since $\wh{h}_{t-r}(X_{t-r})$ has a Gaussian distribution with mean $-m_{t-r}$ and variance $t-r$, its probability density function at $\rmd z + m_{t,r}$ is explicitly given by
\begin{equation}
\begin{multlined}
\frac{1}{\sqrt{2 \pi (t-r)}} \exp \bigg(-\tfrac{ \big(\sqrt{2}(t-r) - \frac{3}{2\sqrt{2}} \log^+(t-r) + m_{t,r} + z \big)^2}{2(t-r)} \, \bigg) \\
\leq C (t-r) \rme^{-(t-r) - \sqrt{2} (z +m_{t,r}) } ,
\end{multlined}
\end{equation}
where we have used the bound on $m_{t,r}$ and the fact that $z \frac{\log^+(t-r)}{t-r} - C' \frac{z^2}{t-r}$ is bounded uniformly in $t,r$ and $z$ for any $C' > 0$. 
At the same time, we can use \eqref{e:52} to bound the conditional probability in \eqref{e:as1} by 
$C (t-r)^{-1} (u^+ + 1)\big((u-z-m_{t,r})^+ + 1\big)$. 

Turning to the first term in~\eqref{e:502}, if $z \leq u$ we use \eqref{e:64} to bound it by
\begin{equation}\label{e:677}
\begin{multlined}
C \Big( \rme^{-r} \rme^{-\sqrt{2} (v-z)} (u-v+1) (u-z + 1) \big(
\rme^{-(v-z)^2/4t} + \rme^{(v-z)/2} \big) \Big)^2 \\
\leq
C \rme^{-2r} (u-v+1)^2 \rme^{-2\sqrt{2} v} \rme^{2\sqrt{2}z} (u-z + 1)^2 
\big(\rme^{-(v-z)^2/4t} + \rme^{v-z}\big) \,.
\end{multlined}
\end{equation}
Otherwise, if $z > u$, we use \eqref{e:64} for one factor and \eqref{e:64.3} for the other. This gives
\begin{equation}\label{e:678}
\begin{multlined}
C \Big(\rme^{-r} \rme^{-\sqrt{2} (v-z)} (u-v+1) \big(
\rme^{-\frac{(v-z)^2}{4t}} + \rme^{\frac{(v-z)}{2}} \big) \Big) \\
\times \Big(\rme^{-r} \rme^{-\sqrt{2} (v-z)} (u-v+1) \rme^{-\frac{3}{2}(z-u)} \Big) \\
=
C \rme^{-2r} (\!u\!-\!v+1)^2 \rme^{-2\sqrt{2} v} \rme^{2\sqrt{2}z} \rme^{-\frac{3}{2}(z-u)}
\big(\rme^{-(v-z)^2/4t} + \rme^{(v-z)/2}\big) \,.
\end{multlined}
\end{equation}

We now split the integral in~\eqref{e:502} according to whether $z \leq u$ or $z > u$. In the former range, we use~\eqref{e:677} and bound it by
\begin{equation}
\label{e:602}
\begin{multlined}
C\rme^{-t-r - \sqrt{2}m_{t,r}} (u^+ +1) (u-v+1)^2 \rme^{-2 \sqrt{2} v} \\
\times
\int_{z \leq u} \Big(\rme^{-\frac{(v-z)^2}{4t} + \sqrt{2}z} + \rme^{v + (\sqrt{2}-1)z} \Big) (u-z+1)^3  \rmd z \,.
\end{multlined}
\end{equation}
Expanding the first parenthesis in the integrand and then integrating each of the resulting terms separately, the integral of the second term is bounded by $C \rme^{\sqrt{2}u-(u-v)}$. For the integral of the first term, we observe that the exponent $-(v-z)^2/(4t) + \sqrt{2}z$ is maximized at $z=2\sqrt{2}t +v$. Therefore, if $u < 2\sqrt{2}(1-\eta)t + v$ for some $\eta > 0$, the integral of the first term is bounded by a constant times the value of the integrand at $u$, which gives the bound $C \rme^{\sqrt{2}u} \rme^{-(u-v)^2/4t}$, with $C>0$ depending on $\eta$.
On the other hand, if $u > 2 \sqrt{2}(1-\eta)t  + v$, then we integrate the first term in absolute value over all $\bbR$, thereby obtaining the upper bound
\begin{equation}
\begin{split}
C t \rme^{2t + \sqrt{2} v}(u-v-2\sqrt{2}t + 1)^3 
\leq C \rme^{\sqrt{2}u} \rme^{-(u-v)/2} \,,
\end{split}
\end{equation}
for $\eta$ small enough, where we have used that $2\sqrt{2}(1-\eta)t \leq u-v$.
Putting all of these together, the integral in \eqref{e:602} can always be bounded by 
\begin{equation}
\begin{multlined}
C \rme^{\sqrt{2}u} \big( \rme^{-(u-v)/2} + \rme^{-(u-v)^2/4t} + \rme^{-(u-v)}\big) \\
\leq C \rme^{\sqrt{2}u} \big(\rme^{-(u-v)^2/4t} + \rme^{-(u-v)/2}\big) \,. 
\end{multlined}
\end{equation}

Returning to the integral in~\eqref{e:502}, in the range $z \geq u$ we use~\eqref{e:678} to the
get the upper bound 
\begin{equation}
\begin{multlined}
\rme^{-t-r - \sqrt{2}m_{t,r}} (u^+ +1) (u-v+1)^2 \rme^{-2 \sqrt{2} v} \\
\times \int_{z \geq u}  \rme^{\sqrt{2}z - \frac{3}{2}(z-u)}
\big(\rme^{-(v-z)^2/4t} + \rme^{(v-z)/2} \big)  \rmd z \,.
\end{multlined}
\end{equation}
The sum of the first two exponents maximizes at $z=v-(3-2\sqrt{2})t  \leq u - (3 - 2\sqrt{2})t$, while the sum of the first and the last exponents always maximizes at $u$. This means that $z=u$ determines the bound on the integral and gives $C \rme^{\sqrt{2}u} \big(\rme^{-(u-v)^2/4t} + \rme^{-(u-v)/2}\big)$ as an upper bound exactly as in the previous range.

Altogether, the integral in~\eqref{e:502} is bounded above by
\begin{equation}
C \rme^{-t-r - \sqrt{2}m_{t,r}} (u^+ +1) (u-v+1)^2 \rme^{-2 \sqrt{2} v}
\rme^{\sqrt{2}u} \big(\rme^{-(v-u)^2/4t} + \rme^{-(u-v)/2}\big) \,.
\end{equation}
To make the identification with the right hand side of~\eqref{e:699} just notice that~\eqref{e:574} implies 
\begin{equation}
\rme^{-\sqrt{2}m_{t,r}} \leq C\big(1+(r \wedge (t-r))^{-3/2}\big) \,,
\end{equation}
proving the statement.
\end{proof}

We can now use the many-to-two lemma to bound the second moment.
\begin{lem}
\label{l:6.4}
There exists $C > 0$ such that for all $v \leq u$,
\begin{equation}
\begin{multlined}
\bbE \Big(\cE_t\big([v, \infty)\big)^2 ;\; \wh{h}^*_t \leq u \Big) \\
\leq \rme^{-2\sqrt{2} v} (u-v+1)^2 \rme^{\sqrt{2}u} (u^+ + 1)  
	 \big(\rme^{-(u-v)^2/4t} + \rme^{-(u-v)/2}\big) \,.
\end{multlined}
\end{equation}
\end{lem}
\begin{proof}
In light of the second equation in~\eqref{e:58.2} we can use (the many-to-two) Lemma~\ref{l:4.2}
with $F(x,y) = 1_{\{ \min \{ \wh{h}_t(y), \wh{h}_t(x) \} \geq v\,,\,\, \wh{h}^*_t \leq u\}}$, thereby obtaining
\begin{equation}
\rme^{3t} \wt{\bbE}^{(2)} \left( \rme^{-\rmd(X_t(1), X_t(2))};\, \min \big\{ \wh{h}_t(X_t(1)), \wh{h}_t(X_t(2)) \big\}  \geq v,\, \wh{h}^*_t \leq u
 \right) \,.
\end{equation}
Conditioning on $\rmd(X_t(1), X_t(2))$ and recalling that the distribution of $t-\rmd \big(X_t(1), X_t(2)\big)$ is exponential with rate $2$ truncated at $t$ (see Remark~\ref{r:4.5}), we may use Lemma~\ref{lem:18.5} to bound the last display by
\begin{equation}
\begin{split}
C & \rme^{\sqrt{2}u}  (u^+ + 1) 
\rme^{-2\sqrt{2} v} (u-v+1)^2 \big(\rme^{-(u-v)^2/4t} + \rme^{-(u-v)/2}\big) \\
& \times \rme^{3t} \bigg(  \rme^{-3t} + \int_{r=0}^t \frac{\rme^{-t-r}}{1+(r \wedge (t-r))^{3/2}} \rme^{-r} \rme^{-2(t-r)} \rmd r \bigg).
\end{split}
\end{equation}
Since the term in the second line is bounded by a constant, the result follows.
\end{proof}

\section{Proofs of Cluster Level Set Propositions}
\label{s:ClusterProperties}
The aim in this section is to prove the cluster properties stated in Subsection~\ref{ss:ClusterProperties}. We start with the following lemma that characterizes the limiting cluster distribution in terms of the cluster around the spine particle, conditioned to be the global maximum.  Recall the spinal decomposition from Subsection~\ref{ss:Spine} and that in particular $X_t$ denots the spine particle at time $t$.
\begin{lem}
\label{l:7.0}
Let $\cC \sim \nu$ be distributed according to the cluster law. Then for any $\nu$-continuity set $B \subseteq \bbM$ and any $u \in \bbR$,
\begin{equation}
\label{e:387}
\bbP(\cC \in B) = \lim_{t \to \infty} 
\wt{\bbP} \big( \cC_{t, r_t} (X_t) \in B \, \big|\, \wh{h}_t(X_t) = \wh{h}_t^* = u\big) ,
\end{equation}
where $\cC_{t, r_t} (X_t):= \sum_{y \in \mathrm{B}_{t,r_t} (X_t) } \delta_{h_t(y) -h_t(X_t)}$ denotes the cluster around $X_t$ as defined in \eqref{e:5B}.
\end{lem}
\begin{proof}
The proof of Theorem~2.3 in \cite{ABBS} shows that $\bbP \big(\cC_{t,r_t}(X_t^*) \in B \big) \longrightarrow \bbP(\cC \in B) $ as $t \to \infty$, where $\cC_{t,r_t}(X_t^*)$ is the cluster around the highest particle $X_t^* : = \argmax_{x \in L_t} h_t(x)$. Thanks to the product structure of the intensity measure governing the limiting Poisson point process and the absolute continuity of its first coordinate, the above limit still holds if we condition on 
$h_t^* = m_t +u$ for any $u \in \bbR$, namely
\begin{equation}
\bbP(\cC \in B) = \lim_{t \to \infty}  \bbP \big(\cC_{t,r_t}(X_t^*) \in B  \,\big|\, h_t^* = m_t +u \big) . 
\end{equation}
We rewrite the probability in the right-hand side above as the conditional expected value of  $\sum_{x \in L_t} 1_{ \{ \cC_{t,r_t} (x)  \in B,\;\wh{h}_t(x)=\wh{h}^*_t\}}$, 
and use (the many-to-one) Lemma~\ref{l:4.1} twice, with 
$x \mapsto 1_{\{\cC_{t,r_t}(x) \in B\} \cap \{ \wh{h}^*_t = \wh{h}_t(x) \in \rmd u\}}$ and then with $x \mapsto 1_{\{ \wh{h}^*_t = \wh{h}_t(x) \in \rmd u\}} $ as the random function $F(x)$, to obtain
\begin{equation}
 \bbP \big(\cC_{t,r_t}(X_t^*) \in B  \mid h_t^* = m_t +u \big) = 
\frac{\wt{\bbP} \big(\cC_{t,r_t}(X_t) \in B \,,\,\, \wh{h}_t(X_t) = \wh{h}_t^* \in \rmd u\big)}{\wt{\bbP} \big(\wh{h}_t(X_t) = \wh{h}^*_t \in \rmd u \big)} \,,
\end{equation}
which is equal to the right hand side of~\eqref{e:387}.
\end{proof}

For what follows in this section, we will mostly work with variants of the conditional probability $\wt{\bbP}(\cdot \mid\wh{h}_t(X_t) = \wh{h}_t^* )$, in which case the configuration $\cC_{t,r_t}(X_t)$ around the spine is exactly the configuration around the maximal particle $X_t^*$ therefore we shorten the notation $\cC_{t,r_t}(X_t)$ into
\begin{equation}
\cC_{t,r_t}^* : = \cC_{t,r_t}(X_t) = \textstyle \sum_{y \in \mathrm{B}_{t,r_t} (X_t) } \delta_{h_t(y) -h(X_t)}
\end{equation}   

We can now begin proving the propositions in Subsection~\ref{ss:ClusterProperties}. We dedicate a subsection to each of these proofs.
\subsection{Proof of Proposition~\ref{p:12} }
The proof of Proposition~\ref{p:12} follows readily from the two results below, whose proofs we postpone to the end of the section. The first one gives the $v \to \infty$ asymptotic of $\wt{\bbE} \, \cC_{t,r_t}^*\big([-v,0]\big)$.

\begin{lem}
\label{l:7.1}
There exists $C > 0$ such that as $t \to \infty$ and then $v \to \infty$,
\begin{equation}
\label{e:389}
\wt{\bbE}\Big(\cC_{t,r_t}^*\big([-v,0]\big) \,\Big|\, \wh{h}^*_t = \wh{h}_t(X_t) = 0\Big)
		\sim C \rme^{\sqrt{2} v} \,.
\end{equation}
\end{lem}

Whereas the second one provides upper bounds for the second moment of $\cC_{t,r_t}^*\big([-v,0]\big)$.
\begin{lem}
\label{l:7.2}
There exists $C > 0$ such that for all $v \geq 0$,
\begin{equation}
\label{e:390}
\limsup_{t \to \infty}
	\wt{\bbE}\Big(\big(\cC_{t,r_t}^*([-v,0])\big)^2 \,\Big|\, \wh{h}^*_t = \wh{h}_t(X_t) = 0\Big)
		\leq C (v+1) \rme^{2\sqrt{2} v} \,.
\end{equation}
\end{lem}

\begin{proof}[Proof of Proposition~\ref{p:12}]
By Lemma~\ref{l:7.2}, for all $v \geq 0$ there exist $t_0 \geq 0$ such that the collection of random variables $\big\{\cC_{t,r_t}^*([-v,0])
:\: t \geq t_0\big\}$ is uniformly integrable under the conditional measure $\wt{\bbP}(\cdot \,|\, \wh{h}^*_t = \wh{h}_t(X_t) = 0\big)$ and therefore in light Lemma~\ref{l:7.0} with $u=0$, the expectation of $\cC^*_{t,r_t}\big([-v,0]\big)$ under this measure converges as $t \to \infty$ to the expectation of $\cC([-v,0])$ under $\nu$, provided that $\cC$ does not charge $-v$ with positive probability. The latter condition, which is equivalent to $[-v, 0]$ being a stochastic continuity set for $\cC$, is needed in order to ensure that $\cC_{t,r_t}^*([-v,0])$ converges weakly to $\cC([-v,0])$ under the conditional measure.

Now, although $[-v, 0]$ is indeed $\cC$-stochastic continuous, we can avoid having to prove this by proceeding in a different way.
Given $v \in \bbR$, we can always find $v - 1/v < v' \leq v \leq v'' \leq v+1/v$ such that $v', v''$ are not charged by $\cC$ with probability $1$. The existence of such points is assured by the fact that the set of points which are charged with positive probability by $\cC$ is at most countable. Then by monotonicity,
\begin{equation}
\begin{split}
\bbE \cC([-v, 0]) &\geq \bbE \cC([-v', 0]) =:\lim\limits_{t \to \infty} \wt{\bbE} \cC^*_{t,r_t}([-v', 0])\,; \\
\bbE \cC([-v, 0]) &\leq   \bbE \cC([-v'',0]) = \lim\limits_{t \to \infty} \wt{\bbE} \cC^*_{t,r_t} \big([-v'', 0]) \,.
\end{split}
\end{equation}
Now, the first and last quantities are asymptotically equivalent to $C\rme^{\sqrt{2}v'}$ and $C \rme^{\sqrt{2}v''}$ respectively, which in light of the choice of $v', v''$ are also asymptotic to 
$C \rme^{\sqrt{2}v}$. This shows the first part of the proposition with $C_\star = C$, where $C$ is the constant in Lemma~\ref{l:7.1}.

The second part of the proposition follows from Lemma~\ref{l:7.0}, Lemma~\ref{l:7.2} and an application of Fatou's lemma, whenever $[-v,0]$ is a stochastic continuity sets under $\cC$ (as a process on $\bbR_-$). As before, if this is not the case, we pick $v''$ as before and use monotonicity again.
\end{proof}

It remains therefore to prove Lemma~\ref{l:7.1} and Lemma~\ref{l:7.2} and at this point we can appeal to Lemma~\ref{l:5.2} to represent the cluster $\cC^*_{t,r_t}$ in terms of the decorated random walk process of Section~\ref{s:Reduction}. This is the content of the next lemma, but before we can state it, we need several new definitions and/or abbreviations. First, recall the random objects: $W$, $H$ and $\cN$ from Section~\ref{s:Reduction} and that $\cE^s_t$ is the extremal process of $h^s_t$. Next, let us abbreviate for $t \geq 0$, 
\begin{equation}
\label{e:694}
\cA_t := \Big \{
\max_{k : \sigma_k \in [0,t]} \big(\wh{W}_{t,{\sigma_k}} + \wh{h}^{\sigma_k*}_{\sigma_k}\big) \leq 0 \Big\} 
\quad, \qquad
\wh{\bbP}_t(\cdot) = \bbP \big(\cdot\,\big|\, \wh{W}_{t,0} = \wh{W}_{t,t} = 0 \big) \,,
\end{equation}
with $\wh{\bbE}_{t}$ the corresponding expectation.
Finally, for $v \geq 0$ and $0 \leq s \leq t$ we set $j_{t,v}(s) := \wh{\bbE}_t J_{t,v}(s)$ where
\begin{equation}
\label{e:394}
J_{t,v}(s):= \cE_s^s \big([-v,0] - \wh{W}_{t,s}\big) 1_{\{\wh{h}^{s*}_s \leq -\wh{W}_{t,s}\}}  \times 1_{\cA_t} \,,
\end{equation}
and for $0 \leq s \leq s' \leq t$, also $j_{t,v}(s,s') := \wh{\bbE}_t J_{t,v}(s,s')$ where
\begin{equation}
\begin{split}
J_{t,v}(s,s'):=  &
\cE_s^s \big([-v,0] - \wh{W}_{t,s}\big) 1_{\{\wh{h}^{s*}_s \leq -\wh{W}_{t,s}\}} \\
& \, \times
\cE_{s'}^{s'} \big([-v,0] - \wh{W}_{t,s'}\big) 1_{\{\wh{h}^{s'*}_{s'} \leq -\wh{W}_{t,s'}\}} \times 1_{\cA_t} \,.
\end{split}
\end{equation}

We now have,
\begin{lem}
\label{l:7.3}
Let $v \geq 0$. Then.
\begin{equation}
\label{e:396}
\wt{\bbE}\Big(\cC_{t,r_t}^*\big([-v,0]\big) ;\; \wh{h}^*_t \leq 0 \,\big|\, \wh{h}_t(X_t) = 0\Big) 
= 2 \int_{s=0}^{r_t} j_{t,v}(s) \rmd s\,.
\end{equation}
and
\begin{equation}
\label{e:397}
\begin{split}
\wt{\bbE} \Big(\big(\cC_{t,r_t}^*([-v,0])\big)^2 ;\; \wh{h}^*_t \leq 0\,\big|\, 
	\wh{h}_t(X_t) = 0\Big) = 4 &\int_{s,s'=0}^{r_t} j_{t,v}(s,s') \rmd s \rmd s' \\
	& + 2 \int_{s=0}^{r_t} j_{t,v}(s,s) \rmd s \,.
\end{split}
\end{equation}
\end{lem}
\begin{proof}
Let us start with~\eqref{e:396}. By Lemma~\ref{l:5.1} with $r=0$, $u=w=0$ and Lemma~\ref{l:5.2} with $r=r_t$ (ignoring the law of $h_{t-s}(X_{t-s})$), we may write the left hand side as
\begin{equation}
\wh{\bbE}_t \int_{s=0}^{r_t} J_{t,v}(s) \cN(\rmd s) \,.
\end{equation}

Since $\cN$ is a Poisson point process on $\bbR_+$ with intensity $2 \rmd x$, it associated Palm kernel can be written as $\big(\bbP(\cN_s \in \cdot) :\: s \geq 0\big)$ where $(\cN_s :\: s \geq 0)$ is a family of point processes such that 
$\cN_s \stackrel{\mathrm{law}}{=} \cN +\delta_s$, assumed to be defined alongside  $W$ and $H$ and independent of them.
Now, conditional on $\cF := \sigma(W, H)$ the random function $J_{t,v}$ depends only on $\cN$  (through the last indicator in its definition). Therefore by Palm-Campbell theorem (see, e.g. Proposition 13.1.IV in ~\cite{daley2007introduction}) and independence between $\cN_s$ and $\cF$,
\begin{equation}
\wh{\bbE}_t \Big(\int_{s=0}^{r_t} J_{t,v}(s) \cN(\rmd s) \,\Big|\, \cF \Big)
= 2 \int_{s=0}^{r_t} \wh{\bbE}_t \big( J_{t,v}(\cN_s,s) \,\big|\, \cF \big) \rmd s  
\ ,\quad \wh{P}_t-\text{a.s.} \,,
\end{equation}
where $J_{t,v}(\cN_s,s)$ is defined as $J_{t,v}(s)$ in~\eqref{e:394} only with $\cN_s$ replacing $\cN$. However, because of the middle indicator in definition~\eqref{e:394}, there is in fact no difference between $J_{t,v}(\cN_s,s)$ and $J_{t,v}(s)$. Taking now expectation with respect to $\wh{\bbE}_t$ and using Fubini's theorem to exchange between the integral and the expectation on the right hand side, we obtain~\eqref{e:396}.

The second claim of the lemma is quite similar. We first write the left hand side of~\eqref{e:397} as
\begin{equation}
\wh{\bbE}_t \Big(\int_{s=0}^{r_t} J_{t,v}(s) \cN(\rmd s)\Big)^2 =
\wh{\bbE}_t \int_{s,s'=0}^{r_t} J_{t,v}(s,s') \cN^2(\rmd s \times \rmd s') \,,
\end{equation}
where $\cN^2$ is the product measure of $\cN$ with itself. Letting $(\cN_{s,s'} :\: s,s' \geq 0)$ be a collection of point process which are independent of $W$ and $H$ and with 
$\cN_{s,s'} \stackrel{\mathrm{law}}{=} \cN +\delta_s + 1_{s' \neq s} \delta_{s'}$, we now use the second order Palm-Campbell Theorem (see, e.g. Ex 13.1.11 in~\cite{daley2007introduction} or alternatively just apply the usual theorem to $\cN^2$). This shows that the last expectation is equal to
\begin{equation}
\int_{s,s'=0}^{r_t} \wh{\bbE}_t \big( J_{t,v}(\cN_{s,s'}, s,s') \big) \cM(\rmd s \times \rmd s') 
= \int_{s,s'=0}^{r_t} \wh{\bbE}_t \big( J_{t,v}(s,s') \big) \cM(\rmd s \times \rmd s') \,,
\end{equation}
where in the first integral $J_{t,v}(\cN_{s,s'}, s,s')$ is defined as $J_{t,v}(s,s')$ only with $\cN_{s,s'}$ replacing $\cN$, again making no difference, and in the second integral $\cM^2$ is the intensity measure of the process $\cN^2$ on $\bbR_+^2$. Since $\cM^2$ satisfies $\cM^2(A) := 4 \int_{s,s'=0}^\infty 1_A(s,s') \rmd s \rmd s' + 2\int_{s=0}^\infty 1_A(s,s) \rmd s$ for all Borel sets $A \subseteq \bbR^2_+$, the result follows.
\end{proof}

Next, we need asymptotics and bounds on $j_{t,v}(s)$ and $j_{t,v}(s,s')$. This is where the results of Section~\ref{s:Moments} will be used. For what comes next, given $M \geq 0$, we shall need the following refinements of $J_{t,v}(s)$ from~\eqref{e:394}:
\begin{equation}
\label{e:901}
J_{t,v}^{< M}(s) = J_{t,v}(s) 1_{\{|\wh{W}_{t,s}| < M\}}
\,, \qquad
J_{t,v}^{\geq M}(s) = J_{t,v}(s) 1_{\{|\wh{W}_{t,s}| \geq M\}} \,,
\end{equation}
with $j_{t,v}^{<M}(s)$, $j_{t,v}^{\geq M}(s)$ the respective expectations under $\wh{\bbE}_t$. 
We start with upper bounds.
\begin{lem}
\label{l:7.4}
There exists $C, C' > 0$ such that for all $t \geq 0$, $0 \leq s \leq t/2$, $v \geq 0$ and $M \geq 0$,
\begin{equation}
\label{e:603}
j_{t,v}^{\geq M}(s) \leq C \frac{  \rme^{\sqrt{2} v} (v+1) }{t (s+1) \sqrt{s}} \times  \rme^{-C' M} \Big(\rme^{- \frac{v^2}{16s}} + \rme^{-\frac{v}{2}}\Big) \,.
\end{equation}
Also, there exists $C > 0$ such that for all $t \geq 0$, $0 \leq s \leq s' \leq t/2$ and $v \geq 0$,
\begin{equation}\label{e:604}
j_{t,v}(s,s')  \leq C \, \frac{ (v+1)^2 \, \rme^{2\sqrt{2}v}  \Big(\rme^{-\frac{v^2}{16s}} + \rme^{-\frac{v}{4}}\Big) \Big(\rme^{-\frac{v^2}{16s'}} + \rme^{-\frac{v}{4}}\Big) }{t (s+1) (s'-s+1) \sqrt{s \, (s'-s +1_{s=s'} )} }     
\end{equation}
\end{lem}

\begin{proof}
Starting with the first inequality, by conditioning on $\wh{W}_{t,s}$ we write $j_{t,v}^{\geq M}(s)$ 
\begin{equation}\label{as.int.1}
\int_{|z| \geq M} q_t\big((0,0);(s,z)\big) \times e_{s,v}(z)\times q_t\big((s,z);(t,0)\big) \times p_t(s,z) \rmd z \,,
\end{equation}
where $q_t\big((s_1,z_1),(s_2,z_2)\big)$ is given by 
\begin{equation}
\label{e:406}
\bbP \Big( \max_{k : \sigma_k \in [s_1,s_2]} \big(\wh{W}_{t,{\sigma_k}} + \wh{h}^{\sigma_k*}_{\sigma_k}\big) \leq 0 \,\Big|\, \wh{W}_{t,s_1} = z_1,\, \wh{W}_{t,s_2} = z_2 \Big);
\end{equation}
$e_{s,v}(z)  := \bbE \big(\cE_s ([-v,0]-z) ;\; \wh{h}^*_s \leq -z \big)$ and $p_t(s,z) $ is the (conditional) density function $\bbP \big(\wh{W}_{t,s} \in \rmd z \,\big|\, \wh{W}_{t,0} = \wh{W}_{t,t} = 0 \big) / \rmd z$.

Observe that the definition of $q_t$ above does not change if we replace $\wh{W}_{t,u}$ by $\wh{W}_{t',u}$ for any $t' \geq s_2$ everywhere inside the probability brackets. Indeed, recalling the definition of $\wh{W}_{t,u}$ in~\eqref{e:20.5}, we see that the difference $\wh{W}_{t,u} - \wh{W}_{t',u}$ is a (deterministic) linear function of $u$, which is lost under the conditioning, because of the Gaussian law of $\wh{W}$. In particular, we can rewrite the integral as
\begin{equation}
\label{e:5.22}
\int_{|z|\geq M} q_s\big((0,0);(s,z)\big) \times e_{s,v}(z)\times q_t\big((s,z);(t,0)\big) \times p_t(s,z) \rmd z \,.
\end{equation}

Now, conditioned on $\wh{W}_{t,0} = \wh{W}_{t,t}$ the law of $\wh{W}_{t,s}$ is Gaussian with mean $-\gamma_{t,s}$ and variance $s(t-s)/t$. Thanks to the assumption $s \leq t/2$, the above variance always lies inside $[s/2, s]$ and hence $p_t(s,z)$ is smaller than 
$C \,s^{-1/2} \rme^{-(z+\gamma_{t,s})^2/2s}$. Using Lemma~\ref{lem:15}, either the first upper bound if $z \leq 0$ or the second if $z \geq 0$, we have
\begin{equation}
q_s\big((0,0);(s,z)\big) \times p_t(s,z) \leq C  \frac{\big(z^- + \rme^{-\frac{3}{2}z^+} \big)}{(s+1) \sqrt{s} } \,.
\end{equation}
Above we have replaced $s^{-1}$ from~\eqref{e:52} by $(s+1)^{-1}$. To justify such replacement we notice that if $s \geq 1$, we can compensate for this change increasing the constant $C$. Whereas, if $s \in [0, 1]$ we just bound the left hand side above by $p_t(s,z)$ which is always smaller than the right hand side, again increasing the constant if necessary.

Using the upper bound in Lemma~\ref{l:18} to estimate $e_{s,v}(z)$ and again the first upper bound in Lemma~\ref{lem:15} for $q_t\big((s,z);(t,0)\big)$, the integral in~\eqref{e:5.22} is smaller than
\begin{multline}
\label{e:511}
C \int_{|z| \geq M}  \frac{\big(z^- + \rme^{-\frac{3}{2}z^+} \big)}{(s+1) \sqrt{s} }  \rme^{\sqrt{2}(v+z)} (v+1) (z^-+1) 
 \Big(\rme^{-\frac{(v+z)^2}{4s}} + \rme^{-\frac{v+z}{2}} \Big) \frac{z^-+1}{t-s}  \rmd z \\
\leq C \frac{  \rme^{\sqrt{2} v} (v+1) }{t (s+1) \sqrt{s}}  \int (z^-+1)^2 \Big(z^- + \rme^{-\frac{3}{2}z^+} \Big) \Big(\rme^{-\frac{(v+z)^2}{4s} + \sqrt{2}z} + \rme^{-\frac{v+z}{2} + \sqrt{2} z} \Big) \rmd z \,,
\end{multline}
where the range of the last integral is $|z|\geq M$.

We now distribute the last parenthesis in the integrand and obtain two distinct integrals. Observing that
$1/2 \leq \sqrt{2} \leq 3/2$, the first integral can be bounded above by $C \rme^{-C' M} \rme^{-v^2/16s}$ if $z \geq -v/2$ and otherwise by 
\begin{equation}
\rme^{-\sqrt{2}((v/2) \vee M)} ((v/2) \vee M + 1)^3 \leq C \rme^{-C' M} \rme^{-v/2} \,.
\end{equation} 
The second can just be bounded by $C\rme^{-C'M} \rme^{-v/2}$. Combining these bounds the last integral in~\eqref{e:511} can always be bounded by $C \rme^{-C'M} \big(\rme^{-v^2/16s} + \rme^{-v/2} \big)$, which shows the first part of the lemma. 

Moving on to the second, assume first that $s \neq s'$ and condition this time on $\wh{W}_{t,s}$ and $\wh{W}_{t,s'}$ to write $j_{t,v}(s,s')$ as
\begin{equation} \label{eq:as.1.123}
\begin{split}
\int_{z,z'} q_s\big((0,0);(s,z)\big) & \times e_{s,v}(z)\times q_{s'}\big((s,z); (s',z')\big) \\
& \times e_{s',v}(z') \times q_t\big((s',z');(t,0)\big) \times p_t \big((s,z); (s',z')\big) \rmd z \rmd z'\,,
\end{split}
\end{equation}
where $p_t\big((s,z);(s',z')\big) = \bbP \big(\wh{W}_{t,s} \in \rmd z, \wh{W}_{t,s'} \in \rmd z' \,\big|\, \wh{W}_{t,0} = \wh{W}_{t,t} = 0 \big)$ and $ e_{\cdot,\cdot}(\cdot)$, $q_\cdot (\cdot)$ are defined as before. Then $p_t\big((s,z);(s',z')\big)$ satisfies
\begin{equation}
p_t\big((s,z);(s',z')\big) \leq \pi^{-1} \bigg( \frac{t}{s(s'-s)(t-s')} \bigg)^{1/2}
\exp \Big(-\tfrac{(z+\gamma_{t,s})^2}{2s} - \tfrac{(z'+\gamma_{t,s'})^2}{2(t-s')} \Big) \,,
\end{equation}

As in the bound for $j_{t,v}(s)$, we now use the upper bounds in Lemma~\ref{lem:15} for both $q_s$ and $q_t$ in the integrand, with the ``right'' bound chosen depending on whether $z$ (respectively $z'$) are positive or negative. This bounds $q_s\big((0,0);(s,z)\big) \times q_t\big((s',z');(t,0)\big) \times p_t\big((s,z);(s',z')\big)$ by
\begin{equation}
C t^{-1} s^{-1/2} (s+1)^{-1} (s'-s)^{-1/2} \big(z^- + \rme^{-\frac{3}{2}z^+}\big) \big(z'^- + \rme^{-\frac{3}{2}z'^+}\big) \,,
\end{equation}
where we have used that $t-s' \in [t/2, t]$ and again replaced the $q_s$ term by $1$ if $s \in [0,1]$.

Using now Lemma~\ref{l:18} to bound the ``$e$-terms'' and again the first upper bound in Lemma~\ref{lem:15} for the remaining ``$q$-term'' if $(s'-s) \geq 1$ or otherwise the trivial bound $1$, the double integral in \eqref{eq:as.1.123} is bounded up to a multiplicative factor by
\begin{equation}
\begin{split}
\tfrac{ \rme^{2\sqrt{2} v} (v +  1)^2}{t  (s  +  1) \! (s' \! -  s \!+ 1) \! \sqrt{ \! s (s' \! - s)}}  
 & \! \int_{z,z'} \textstyle \! \Big(\! z^- \! + \! \rme^{-\frac{3}{2}z^+} \! \Big) \! \Big( \! z'^- \!+ \! \rme^{-\frac{3}{2}z'^+} \!\Big)
(\! z^- \! + \! 1 \!)^2( \! z'^- \! + \!1 \!)^2   \\
& \, \ \textstyle \times \Big(\! \rme^{-\frac{(v+z)^2}{4s}} \! + \! \rme^{-\frac{(v+z)}{2}} \! \Big)\!\Big( \!\rme^{-\frac{(v+z')^2}{4s'}}  \! + \! \rme^{-\frac{(v+z')}{2}} \! \Big)\!  \rmd z  \rmd z' .
\end{split}
\end{equation}
The above integral factors into two identical single variable integrals which are again equal to the integral in~\eqref{e:511} when $M=0$. Therefore the bound obtained there applies making the double integral smaller than 
$\big(\rme^{-v^2/(16s)} + \rme^{-v/2} \big)\big(\rme^{-v^2/(16s')} + \rme^{-v/2} \big)$ and the whole last display smaller than the right hand side of~\eqref{e:604}.

Lastly we handle the case $s=s'$ and it is here where we need the second moment bound from Section~\ref{s:Moments}. Again, we write $j_{t,v}(s,s)$ as
\begin{equation}
\int_{z} q_s\big((0,0);(s,z)\big) \times e^{(2)}_{s,v}(z)\times q_t\big((s,z);(t,0)\big) \times p_t(s,z) \rmd z \,,
\end{equation}
where $e^{(2)}_{s,v}(z) := \bbE \big(\cE \big([-(v+z),-z] \big)^2 ;\; \wh{h}^*_s \leq -z \big)$.
We now repeat the argument in the proof of~\eqref{e:603} with $M=0$, only that we use the bound on $e^{(2)}_{s,v}(z)$ from Lemma~\ref{l:6.4} instead of the bound on $e_{s,v}(z)$. This gives as an upper bound on $j_{t,v}(s,s)$,
\begin{equation}
C \, \frac{ 1 }{t \sqrt{s} (s+1)} 
(v+1)^2\rme^{2\sqrt{2} v} \Big(\rme^{-\frac{v^2}{4s}} + \rme^{-\frac{v}{2}}\Big)
\int_z (z^-+1)^2 \big(z^- + \rme^{-\frac{3}{2}z^+} \big) \rme^{\sqrt{2}z} \rmd z \,.
\end{equation}
The last integral is bounded by a constant and thus the whole expression can be made smaller than the right hand side of~\eqref{e:604} if we properly tune the preceding constants.
\end{proof}

Next, we need also asymptotics for $j_{t,v}^{<M}(s)$. This is given in the next lemma
\begin{lem}
\label{l:7.5}
There exists $C > 0$ such that as $t \to \infty$ followed by $v \to \infty$ and then $M \to \infty$,
\begin{equation}
j_{t,v}^{<M}(s) \sim C t^{-1} s^{-3/2} v \rme^{\sqrt{2} v - \frac{v^2}{2s}} \,,
\end{equation}
uniformly in $s \in [\eta v^2, v^2/\eta]$ for any fixed $\eta >0$.
\end{lem}
\begin{proof}
As in the previous lemma, we start by writing $j_{t,v}^{<M}(s)$ as the integral
\begin{equation}
\label{e:418}
j_{t,v}^{<M}(s) = \int_{|z| < M} q_s\big((0,0);(s,z)\big) \times e_{s,v}(z)\times q_t\big((s,z);(t,0)\big) \times p_t(s,z) \rmd z \,,
\end{equation}
with $q_s(\cdot), e_{s,v} (\cdot)$ and $p_t(\cdot)$ defined as before. We now use the corresponding asymptotic results, in place of the upper bounds we have used before, to derive asymptotics for the above integral when the limits are taken in the prescribed order. 

Accordingly, let us first fix $\eta$, $s$, $v$ and $M$ and take $t \to \infty$. 
Conditioned on $\wh{W}_{t,0} = \wh{W}_{t,t} = 0$, the law of $\wh{W}_{t,s}$ is Gaussian with mean $\frac{3}{2\sqrt{2}}\big(s(\log^+ t)/t - \log^+ s\big)$ and variance $s(t-s)/t$. Hence, for all $z$ and $s$ fixed the density $p_t(s,z)$ of $\wh{W}_{t,s}$ tends to 
\begin{equation}
(2\pi s)^{-1/2} \exp \left(-\tfrac{\big(z+\frac{3}{2\sqrt{2}}\log^+s\big)^2}{2s} \right) \qquad \text{as $t \to \infty$},
\end{equation} 
and is bounded by $C s^{-1/2}$ for all $t \geq s/2$ and any $z \in \bbR$ fixed. At the same time, by the third part of Lemma~\ref{lem:15}, we know that $q_t((s,z);(t,0))$ is asymptotic equivalent to $2t^{-1} f^{(s)}(z) g(0)$ as $t \to \infty$. The first upper bound in the same lemma also says that $q_t((s,z);(t,0))$ is smaller than $C (t-s)^{-1}(z^-+1) < 2 C t^{-1}(z^-+1)$ if $t \geq s/2$, which yields $f^{(s)}(z) \leq C(z^-+1)$ for all $s >0, \, z \in \bbR$ and $t$ sufficiently large.
Then, using the dominated convergence theorem, we can replace the quantities in the integrand of~\eqref{e:418} with their asymptotic equivalences and obtain that the integral itself is asymptotic to
\begin{equation}
\label{e:619}
2  \frac{g(0)}{t  \sqrt{2\pi s}} \int_{|z| < M} q_s\big((0,0);(s,z)\big) e_{s,v}(z) f^{(s)}(z) \exp \Big(-\tfrac{\big(z+\frac{3}{2\sqrt{2}}\log^+s\big)^2}{2s} \Big) \rmd z \,,
\end{equation}
when $t \to \infty$ for fixed $s$ and $v$.

Next, we  keep $M$ fixed and take $v \to \infty$. We will consider $s \in [\eta v^2, \eta^{-1}v^2]$, so that $s \to \infty$ as well. Then, by the third part of Lemma~\ref{lem:15} again, we have that for any fixed $z$
\begin{equation}
q_s\big((0,0);(s,z)\big)  \sim 2 \, \frac{f^{(0)}(0)g(z)}{s} 
\qquad \text{as $s \to \infty$},
\end{equation}
with $f^{(0)}(0), g(z) > 0$ from the lemma. Moreover, the upper bounds in the same lemma also show that the left hand side above is smaller than $C s^{-1}(z^-+\rme^{-3z^+/2})$ for all $z$ and $s$. Again, this implies that $g(z) \leq C(z^-+\rme^{-3z^+/2})$ for all $z$ with the constant independent of $s$.
As for $f^{(s)}(z)$, the last part of Lemma~\ref{lem:15} says that $f^{(s)}(z)$ is positive and it tends to $f(z) > 0$ as $s \to \infty$ and since we have established that $f^{(s)}(z) \leq C(z^-+1)$ the same bound applies to the function $f$. 
Finally, we estimate $e_{s,v}(z)$ using Lemma~\ref{l:18} with $u, v, t$ there replaced by $-z, -(v+z)$ and $s$ respectively. Since $|z| \leq M$ and $ \eta \sqrt{s} \leq  v \leq  \eta^{-1} \sqrt{s}$, the conditions of the lemma are satisfied with $\epsilon= 1/M$ and all $s$ large enough, which yields



\begin{equation}
e_{s,v}(z) \sim v  \frac{g(z)}{\sqrt{\pi}}  \exp \Big( \sqrt{2}(v+z) - \tfrac{(v+z)^2}{2s} \Big) 
\sim 
 \Big( v \rme^{\sqrt{2} v - \frac{v^2}{2s}} \Big)  \frac{g(z)}{\sqrt{\pi}} \rme^{\sqrt{2} z} \,,
\end{equation}
when $v \to \infty$ uniformly in $s \in [\eta v^2, \eta^{-1}v^2]$ and $|z| < M$.
Combining all the above and using the dominated convergence theorem again, we see that the integral in~\eqref{e:619} is asymptotic to
\begin{equation}
C   \frac{v \rme^{\sqrt{2} v - \frac{v^2}{2s}}}{s} \int_{|z|<M} \rme^{\sqrt{2}z} f(z) g(z)^2 \rmd z \,,
\end{equation}
as $v \to \infty$ uniformly in $s$ as required and for fixed $M$. Finally, in light of the positivity and upper bounds for $f$ and $g$ the last integral converges when $M \to \infty$ to a positive and finite constant. Collecting all the results together, we finish the proof.
\end{proof}

We can now prove Lemma~\ref{l:7.1} and Lemma~\ref{l:7.2} and thereby complete the proof of Proposition~\ref{p:12}.
\begin{proof}[Proof of Lemma~\ref{l:7.1}]
Fix first $v \geq 0$ and write  $\wt{\bbE}\big(\cC_{t,r_t}^*\big([-v,0]\big) \,\big|\, \wh{h}^*_t = \wh{h}_t(X_t) = 0\big)$ as
\begin{equation}
\frac{\wt{\bbE}\Big(\cC_{t,r_t}^*\big([-v,0]\big) ;\; \wh{h}^*_t \leq 0 \,\Big|\, \wh{h}_t(X_t) = 0\Big)}{\wt{\bbP} \big(\wh{h}^*_t \leq 0 \,\big|\, \wh{h}_t(X_t) = 0\big)} \,.
\end{equation}
An application of Lemma~\ref{l:5.1} with $r=u=w=0$ followed by the third part of Lemma~\ref{lem:15} 
shows that the denominator is asymptotic to $C t^{-1}$ as $t \to \infty$ with $C \in (0,\infty)$. Hence, it remains to treat the numerator.

Now let $M, \eta > 0$, assume $t$ is large enough and use Lemma~\ref{l:7.3} to write the numerator as  
\begin{equation}
\label{e:824}
2\int_{s=\eta v^2}^{\eta^{-1}v^2} j^{< M}_{t,v}(s) \rmd s +
2\int_{s=0}^{r_t} \Big(j_{t,v}(s) 1_{s\in[\eta v^2,  \eta^{-1} v^2]^\rmc} +
j^{\geq M}_{t,v}(s) 1_{s\in[\eta v^2,  \eta^{-1} v^2]}  \Big) \rmd s \,.
\end{equation}
We first want to claim that the second integral becomes negligible when $M \to \infty$ and $\eta \to 0$, in the asymptotic regime we consider. To this end, we observe that $j_{t,v}(s)=j_{t,v}^{\geq 0}(s)$, so the first upper bound in Lemma~\ref{l:7.4} may be used to estimate $j_{t,v}^{\geq M}(s)$ as well as $j_{t,v}(s)$ and bound the second integral above by $C t^{-1} \rme^{\sqrt{2} v} (v+1)$ times
\begin{equation}
\label{e:924}
\begin{split}
& \int_{s=0}^\infty \frac{\rme^{-v^2/16s} + \rme^{-v/2}}{ \sqrt{s} (s+1)}  
\Big(1_{ \{ s \in [\eta v^2, \eta^{-1}v^2]^\rmc \}} + \rme^{-C' M} 1_{\{ s \in [ \eta v^2, \eta^{-1}v^2] \}} \Big) \rmd s \\
& \ 
\leq \int_{0}^\infty  \! \frac{\rme^{-\frac{v}{2}}}{ \sqrt{s} (s\!+\!1)}  \rmd s 
\! +\! \int_{0}^{\eta v^2} \!  \frac{\rme^{-\frac{v^2}{16s}}}{ s^{3/2}}  \rmd s  
\!+ \! \rme^{-C'M} \! \int_{\eta v^2}^\infty  s^{-\frac{3}{2}}  \rmd s 
\!+\! \int_{\frac{v^2}{\eta} }^\infty  s^{-\frac{3}{2}}  \rmd s  \\
& \ \leq 
C \bigg(\rme^{-\frac{v}{2}} + \frac{1}{v}  + \frac{\rme^{-C' M}}{v  \sqrt{\eta} } + \frac{ \sqrt{\eta}}{v} \,  \bigg) \,.
\end{split}
\end{equation}
Therefore the second integral is bounded above by $t^{-1} \rme^{\sqrt{2}v}$ times $\big(\rme^{-v/4} + \rme^{-C' M} \eta^{-1/2} + \sqrt{\eta}\big)$. The latter factor tends to $0$ when $v \to \infty$ followed by $M \to \infty$ and then $\eta \to 0$.

At the same time, thanks to the uniform convergence in Lemma~\ref{l:7.5} we know that as $t \to \infty$ followed by $v \to \infty$ and then $M \to \infty$, the first integral in~\eqref{e:824} is asymptotic equivalent to 
\begin{equation}
C t^{-1} v \rme^{\sqrt{2} v} \int_{s=\eta v^2}^{\eta^{-1}v^2 } s^{-3/2} \rme^{-v^2/2s} \rmd s
= C t^{-1} \rme^{\sqrt{2} v} \int_{y=\eta}^{\eta^{-1}} y^{-3/2} \rme^{-1/2y} \rmd y \,,
\end{equation}
where we have substituted $y=v^2 s$ to obtain the second integral. Taking now $\eta \to 0$, the last integral converges to a constant which is positive and finite.

Combining the estimate on the first integral with the bound on the second shows that the numerator is asymptotically equivalent to $C t^{-1} \rme^{\sqrt{2}v}$ as $t \to \infty$ followed by $v \to \infty$. Together with the $Ct^{-1}$ asymptotics for the denominator, this yields the desired result. 
\end{proof}

Lastly, we provide:
\begin{proof}[Proof of Lemma~\ref{l:7.2}]
As in the proof of Lemma~\ref{l:7.1} we can write the left hand side of~\eqref{e:390} as 
\begin{equation}
\frac{\wt{\bbE}\Big(\big(\cC_{t,r_t}^*([-v,0])\big)^2 ;\; \wh{h}^*_t \leq 0 \,\big|\, \wh{h}_t(X_t) = 0\Big)}{ \wt{\bbP} \big( \wh{h}^*_t\leq 0 \vert  \wh{h}_t(X_t) =0\big) } \,.
\end{equation}
The denominator is asymptotic to $C t^{-1}$ and hence it is enough to show that the expectation in the numerator is bounded above by $Ct^{-1} (v+1)\rme^{2\sqrt{2}v}$ for all $t$ large enough. Again, we can use Lemma~\ref{l:7.3} and Lemma~\ref{l:7.4} to bound this expectation for fixed $v$ and $t$ large enough by $C t^{-1} \rme^{2 \sqrt{2}v} (v+1)^2$ times
\begin{equation}
\int_{s=0}^{\infty} \frac{\rme^{-v^2/16s} + \rme^{-v/4}}{ \sqrt{s} (s+1)} \times
\bigg(1 + \int_{s'=s}^{\infty} \frac{1}{\sqrt{s'-s}(s'-s+1)} \rmd s'\bigg) \rmd s\,.
\end{equation}

The second integral is bounded by a constant uniformly in $s$. The first is bounded by a constant if $v \in [0,1]$ and otherwise, using the substitution $s=v^2 y$, by 
\begin{equation}
C \rme^{-v/4} + v^{-1} \int_{y=0}^\infty y^{-3/2} \rme^{-1/16y} \rmd y 
\leq C (v+1)^{-1} \,.
\end{equation}
All together the expectation in question is bounded above by $C t^{-1} \rme^{2 \sqrt{2}v} (v+1)$ whenever $t$ is large, as we set out to prove.
\end{proof}

\subsection{Proof of Proposition~\ref{prop:right.tail.cluster}}
\begin{proof}[Proof of Proposition~\ref{prop:right.tail.cluster}]
As in the proofs before, by monotonicity it is enough to show that the limit in~\eqref{equation.thm:right.tail.cluster.BBM} holds along $v$'s which are not charged with positive probability by $\cC$. Assuming that $v$ is as such, we use Lemma~\ref{l:7.0} with $u=0$, Lemma~\ref{l:5.1} with $r=u=w=0$ and finally Lemma~\ref{l:5.2} to write $\bbP \big(\cC([-v, 0)) = 0 \big)$ as the limit 
\begin{equation}
\label{e:96B}
\lim_{t \to \infty}  
\frac{\bbP \Big(\max \limits_{k: \sigma_k \in [0,t]} \Big(\wh{W}_{t, \sigma_k} \!+ \! \wh{h}^{\sigma_k*}_{\sigma_k} \! + \!v 1_{[0,r_t]}(\sigma_k) \Big)\! \leq \! 0
 \,\Big|\, \wh{W}_{t,0} \! = \! \wh{W}_{t,t} \! = 0 \Big)}
{\bbP \Big(\max\limits_{k: \sigma_k \in [0,t]} \big(\wh{W}_{t, \sigma_k} + \wh{h}^{\sigma_k*}_{\sigma_k} \big) \leq 0 \,\Big|\, \wh{W}_{t,0} = \wh{W}_{t,t}  = 0 \Big)} \,.
\end{equation}
The denominator is asymptotic to $C t^{-1}$ by the third statement in Lemma~\ref{lem:15} with $v=w=0$ and $r=r_t$. It therefore remains to bound the numerator.

For a lower bound, we follow the heuristics of Brunet and Derrida and restrict the event in the numerator by intersecting with the event that up time $v/2$ there was no branching and that at this time  $\wh{W}_{t,v/2} \leq -v$. Explicitly, we lower bound the numerator in~\eqref{e:96B} by
\begin{equation}
\label{e:97B}
\begin{multlined}
\bbP \big( \sigma_1 > v/2 ,\, \wh{W}_{t,v/2} \leq -v \,\big|\, \wh{W}_{t,0} = \wh{W}_{t,t} = 0 \big) \\
\times
\bbP \Big(\max_{k: \sigma_k \in [0,t]} \big(\wh{W}_{t, \sigma_k} + \wh{h}^{\sigma_k*}_{\sigma_k} +v1_{[0,r_t]}(\sigma_k) \big) \leq 0
 \,\Big|\, \wh{W}_{t,v/2} =-v  ,\,  \wh{W}_{t,t}  = 0 \Big) \,,
\end{multlined}
\end{equation}
where we have used the stochastic monotonicity of the trajectories of $\wh{W}_{t,s}$ with respect to the initial conditions in the second term above. Now $\wh{W}_{t,v/2}$ under $\wh{W}_{t,0} = \wh{W}_{t,t} = 0$ has a Gaussian law with mean $\frac{3}{2\sqrt{2}}\big(v (2t)^{-1} \log^+t -\log^+ (v/2) \big)= -\frac{3}{2\sqrt{2}}\log^+(v/2) + o(1)$ and variance $v(2t-v)/(4t) = v/2 + o(1)$, with both $o(1)$ terms tending to $0$ as $t \to \infty$. At the same time $\sigma_1$ is exponential with rate $2$ and independent of $\wh{W}$. It follows therefore that the first probability in~\eqref{e:97B} will be bounded from below by
\begin{equation}
C \rme^{-v} \tfrac{1}{(v/2)^{1/2}} \exp \Big(-\tfrac{ ( \frac{3}{2\sqrt{2}}\log^+(v/2) +v )^2}{v} \Big) 
\geq C' v^{-\frac{1}{2}} \rme^{-2v} \,,
\end{equation}
for all $t$ large enough.

As for the second probability in~\eqref{e:97B}, using the total probability formula with respect to $W_{t,r_t}$ and recalling the definition of $q_t$ from~\eqref{e:406}, it is at least
\begin{equation}
\begin{split}
\int_{w=-r_t^{2/3}}^{-2v} 
	q_t & \big((v/2, 0);(r_t, w\!+\!v)\big)  \times  q_t\big((r_t, w);(t,0)\big) \\
	&\times \bbP \big(\wh{W}_{t,r_t} \in \rmd w \,\big|\, \wh{W}_{t,v/2} = -v,\, \wh{W}_{t,t} = 0 \big) \,\rmd w .
	\end{split} 
\end{equation}
As we have noticed before (e.g. in the proof of Lemma~\ref{l:7.4}), we can replace $W_{t,u}$ by $W_{r_t,u}$ in the definition of $q_t$, thereby obtaining,
\begin{equation}\label{li.C10}
q_t\big((v/2, 0);(r_t, w+v)\big)  = q_{r_t} \big((v/2, 0);(r_t, w+v)\big)\,.
\end{equation}
Thanks to the asymptotic statement in Lemma~\ref{lem:15}, the right hand side  of \eqref{li.C10} is at least $C r_t^{-1}w^-$ in the above ranges of $v,w$ for all $t$ large enough. The same statement also shows that $q_t\big((r_t, w);(t,0)\big) \geq C' t^{-1}w^-$ under the same conditions.

With the bounds above replacing the corresponding quantities, the last integral is equal to
\begin{equation}
\frac{C}{t r_t } \bbE \Big(\wh{W}_{t,r_t}^2\; ;\; \wh{W}_{t,r_t} \in \big[-(r_t)^{2/3},\, -2v\big] \,\Big|\, \wh{W}_{t,v/2} = -v,\, \wh{W}_{t,t} = 0 \Big) \,.
\end{equation}
Under the conditioning $\wh{W}_{t,r_t}$ is Gaussian with mean and variance given respectively~by
\begin{equation}
\begin{split}
(\gamma_{t,v/2} -v ) &\frac{t-r_t}{t-\tfrac{v}{2}} -\gamma_{t,r_t} = \tfrac{3}{2\sqrt{2}}\big(\log \frac{v}{2} - \log r_t - v+ o(1)\big)  \\
& \quad \text{ and } \quad
\big(r_t - \tfrac{v}{2} \big) \frac{t- r_t}{t-\tfrac{v}{2}} = r_t - \tfrac{v}{2} + o(1) \,,
\end{split}
\end{equation}
with $o(1) \to 0$ as $t \to \infty$. Therefore, for all $t$ large enough the last expectation is at least $C r_t$, making the entire expression bounded below by $C t^{-1}$. Plugging this in~\eqref{e:96B} shows that the numerator is at least $C' t^{-1}v^{-1/2} \rme^{-2v}$ and in light of the asymptotics for the denominator, also that for all $v \geq 1$,
\begin{equation}
\label{e:103B}
\bbP \big(\cC([-v, 0)) = 0 \big) \geq C'v^{-1/2} \rme^{-2v} \,.
\end{equation}

We turn to an upper bound for the numerator of~\eqref{e:96B}. Thanks to Lemma~\ref{l:TailBBMMax}, we know that the lower tails of $\wh{h}^*_t$ decay uniformly in $t \geq 0$. It follows that for any $\epsilon > 0$, there must exists $M > 0$ large enough, such that $\bbP(\wh{h}^*_t < -M) < \epsilon$. Fixing such $\epsilon > 0$ and $M$ and assuming that $v > M$ and that $t \geq r_t^2$, we let 
\begin{equation}
\tau = \inf \big\{s \geq 0 :\: \wh{W}_{t,s} = -v + M\} \wedge v^2\big\} \,.
\end{equation}
Then the numerator in~\eqref{e:96B} conditional on $\tau = s \leq v^2$, is at most 
\begin{equation}
\begin{multlined}
\bbP \Big(\max_{k:\: \sigma_k \in [0,s]} \wh{h}^{\sigma_k*}_{\sigma_k} \leq -M\Big) \\
\times \bbP \Big(\max_{k: \sigma_k \in [s,t]} \big(\wh{W}_{t, \sigma_k} + \wh{h}^{\sigma_k*}_{\sigma_k} \big) \leq 0  \,\Big|\, \wh{W}_{t,s} = -v,\, \wh{W}_{t,t}  = 0 \Big) \,,
\end{multlined}
\end{equation}
where we have used stochastic monotonicity of $W$ with respect to the boundary conditions and independence between $W$, $H$ and $\cN$.

Conditioning on $\cN([0,s])$ and using the fact that $\bbP ( \wh{h}^{\sigma_k*}_{\sigma_k} \leq -M ) < \epsilon$ for each atom $\sigma_k$ of $\cN$ under the conditioning, we may bound the first probability above by $\bbE \epsilon^{\, \cN([0,s])} = \rme^{-2s(1-\epsilon)}$. 
As for the second, using the first upper bound in Lemma~\ref{lem:15}, we see that it is bounded above by $C(v+1)(t-s)^{-1} \leq C'(v+1)t^{-1}$, for all $t$ large enough with $C'$ not depending on $v$.  

At the same time, conditional on $\wh{W}_{t,0} = \wh{W}_{t,t} = 0$ the distribution of $\wh{W}_{t,s}$ is Gaussian with mean $-\gamma_{t,s} = -\frac{3}{2\sqrt{2}}\log^+s + o(1)$ and variance $s(t-s)t^{-1} = s + o(1)$ as $t \to \infty$ with both $o(1)$ tending to $0$ uniformly in $s \leq v^2$. Then, setting $z:= -v + M$, for all $v$ large enough and then $t$ large enough we have
\begin{equation}
\begin{split}
\bbP \big(\tau \! \in \! \rmd s \,\big|\, \wh{W}_{t,0} \!= \! \wh{W}_{t,t}  \!= \!0 \! \big) /\rmd s & 
\leq \bbP \big(\wh{W}_{t,s} \in \rmd z \,\big|\, \wh{W}_{t,0} = \wh{W}_{t,t} =0 \big) / \rmd z   \\ 
& \leq C s^{-1} \exp\Big(-  \tfrac{\big(v - M -\tfrac{3}{2\sqrt{2}}\log^+s\big)^2}{2s}\Big) \\
&\leq C s^{-1} \exp \big(- (1-\epsilon) v^2/(2s) \big) \,, 
\end{split}
\end{equation}
whenever $s < v^2$.

Collecting the above bounds and using the total probability formula, we see that the probability of the event in the numerator of~\eqref{e:96B} is bounded above by
\begin{equation}
C \frac{(v+1)}{t} \Big( \int_{s=0}^{v^2} s^{-1} \rme^{-(1-\epsilon)\big(2s + \frac{v^2}{2s}  \big)} \rmd s 
+ \rme^{-2(1-\epsilon) v^2} \Big) \,. 
\end{equation}
The exponent in the integrand is maximized at $s=v/2$, and its value then is  $-2(1-\epsilon)v$. The last display is therefore at most $C t^{-1}\rme^{-2(1-2\epsilon)v}$ for all $v$ large enough. Together with the asymptotics for the denominator in~\eqref{e:96B}, this shows that for any $\epsilon > 0$ if $v$ is large enough, then
\begin{equation}
\label{e:108B}
\bbP \big( \cC([-v,0)) = 0 \big) \leq C \rme^{-2(1-2\epsilon)v} \,,
\end{equation}

Combining~\eqref{e:103B} with~\eqref{e:108B} shows what we wanted to prove.
\end{proof}

\section{Proofs of Extreme Level Set Theorems}
\label{s:GlobalProperties}

In this section we combine the results concerning cluster properties from the previous section with the law of the limiting generalized extremal process $\wh{\cE}$ to derive asymptotic results for $\cE$. We then use the convergence of the finite time generalized extremal process $\wh{\cE}_t$ to its corresponding limit, to derive asymptotic statements for the extremal level sets of $h$.

\subsection{Structure of Extreme Level Sets}
We start with a lemma that essentially contains the statement of Theorem~\ref{thm.asyden} and Theorem~\ref{t:14}. Recall the definition of $\cE(\cdot ;\; B)$ and $\cE_t(\cdot ;\; B)$ in~\eqref{e:421}.
\begin{lem}
\label{l:8.1}
Let $C_\star$ be as in Proposition~\ref{p:12} and $Z$ be as in~\eqref{e:303}. Then, for all $\alpha \in (0,1]$ as $v \to \infty$, 
\begin{equation}
\label{e:537}
\frac{\cE\big([-v,\infty) ;\; [-\alpha v, \infty) \big)}
{C_\star Z v \rme^{\sqrt{2}v}} \overset{\bbP} \longrightarrow \alpha \,.
\end{equation}
\end{lem}

\begin{proof}
Given $-\infty < -v < w < z \leq \infty$, let us abbreviate 
\begin{equation}
\label{e:43}
\textstyle
F_v (w,z) := \cE \big([-v, \infty) ;\; [w, z] \big) = \sum_{(u, \cC) \in \wh{\cE}} \ \cC \big( [-v - u, 0] \big) 1_{[w, z]}(u)  \,.
\end{equation}
Since conditional on $Z$ the law of $\wh{\cE}$ is that of a Poisson point process whose intensity factorizes (see \eqref{e:N7}), we can write
\begin{gather}
\bbE \big( F_v(w,z) \,\big|\, Z \big) = \int_w^z \bbE \cC \big([-v-u, 0]\big) Z \rme^{-\sqrt{2} u} \rmd u \,, \\
\Var \big( F_v(w,z) \,\big|\, Z \big) = \int_w^z \bbE \big(\cC \big([-v-u, 0]\big)\big)^2 Z \rme^{-\sqrt{2} u} \rmd u \,, 
\end{gather}
with $\cC$ distributed according to $\nu$. Using then Proposition~\ref{p:12}, observing that the right hand side in the first statement of the proposition can be made into an upper bound, albeit with a different constant, we then get
\begin{align}
\label{e:41}
\bbE \big( F_v(w,z) \,\big|\, Z \big)
& \leq C \int_w^z \rme^{\sqrt{2}(v+u)} Z \rme^{-\sqrt{2} u} \rmd u 
= C Z \rme^{\sqrt{2}v} (z-w) \,, \\
\label{e:42}
\Var \big( F_v(w,z) \,\big|\, Z \big)
&  \leq C \int_w^z (v+u) \rme^{2\sqrt{2}(v+u)} Z \rme^{-\sqrt{2} u} \rmd u  \\
& \leq C' Z \rme^{2 \sqrt{2} v + \sqrt{2} z} (z+v), \notag ,
\end{align}
which is valid for all $v,w,z$ as above. Moreover,
\begin{equation}
\label{e:43.5}
\bbE \big( F_v(w,z) \,\big|\, Z \big)
\sim C_\star Z \rme^{\sqrt{2}v} (z-w) \,,\quad\mbox{as $w+v \to \infty$ and uniformly in $z$.}
\end{equation}

Now given $\alpha$ as in the conditions of the Proposition and $v \geq 1$, let us set $w = -\alpha v$, $u = -\alpha v + \sqrt{\log v}$ and $z = \sqrt{\log v}$ and write
\begin{equation}
\label{e:44.5}
F_v(w,\infty) = F_v(w,u) + F_v(u,z) + F_v(z,\infty) \,.
\end{equation}
For the first term, we obtain from~\eqref{e:41} that 
\begin{equation}
\frac{\bbE \big( F_v(w,u) \,\big|\, Z \big)}{Zv\rme^{\sqrt{2}v}}
\leq C \frac{\sqrt{\log v}}{v} \, \overset{v \to \infty} \longrightarrow 0 \,,
\quad \text{for all $v \geq 1$,}
\end{equation}
which implies by Markov's inequality that  
$F_v(w,u) / \big(Zv\rme^{\sqrt{2}v}\big)$ 
converges to $0$ as $v \to \infty$ in $\bbP(\cdot|Z)$-probability for almost every $Z$ and hence that $F_v(w,u) / \big(Zv\rme^{\sqrt{2}v}\big)$ converge to $0$ in $\bbP$-probability. 
As for the second term in~\eqref{e:44.5}, we use respectively~\eqref{e:43.5} and~\eqref{e:42} to obtain 
\begin{equation}
\begin{split}
& \frac{ \bbE \big( F_v(u,z) \,\big|\, Z \big) }{C_\star Z v \rme^{\sqrt{2}v} } 
\sim \frac{z-u}{v} \,  \sim \alpha  \qquad \text{as $v \to \infty$};\\
& \qquad \text{and} \qquad
\frac{ \Var \big( F_v(u,z) \,\big|\, Z \big) }{ \bbE \big( F_v(u,z) \,\big|\, Z \big)\big)^2 }
\leq \frac{C \rme^{\sqrt{2} z}}{Z} \frac{ (z+v) }{ (z-u)^2} 
\, \overset{v \to \infty} \longrightarrow 0 \,.
\end{split}
\end{equation}
Chebyshev's inequality then shows that $\big(C_\star Zv\rme^{\sqrt{2}v}\big)^{-1} F_v(u,z)$ tends to $\alpha$ as $v \to \infty$ in $\bbP(\cdot\,|\,Z)$-probability for almost every $Z$ and hence that 
\begin{equation}
 \big(C_\star Zv\rme^{\sqrt{2}v}\big)^{-1} F_v(u,z)   \to \alpha \quad \text{in $\bbP$-probability as $v \to \infty$}.
\end{equation}

Lastly, for the third term in~\eqref{e:44.5}, observe that whenever $\wh{\cE}([z, \infty) ) = 0$, we also have $F_v(z,\infty) = 0$. Since conditional on $Z$, the intensity measure governing the law of $\wh{\cE}$ is finite on $[0,\infty)$ almost surely, the latter must happen for large enough $z$. This shows that $F_v(z,\infty) \, \overset{v \to \infty} \longrightarrow 0$ almost surely and in particular that 
\begin{equation}
\label{e:52.5}
F_v(z,\infty) \big(Zv\rme^{\sqrt{2}v}\big)^{-1} \longrightarrow 0
\qquad \text{as $v \to \infty$ in $\bbP$-probability.}
\end{equation}
Combining the convergence results for the three terms in the left hand side of~\eqref{e:44.5} shows that
$F_v(w,\infty)  \big(C_\star Zv\rme^{\sqrt{2}v}\big)^{-1}$ converges in $\bbP$-probability to $\alpha$ as $v \to \infty$. Since $F_v(w,\infty)$ is precisely the left hand side of~\eqref{e:37}, the proof is complete.
\end{proof}

The proof of Theorem~\ref{thm.asyden} and Theorem~\ref{t:14} are now straightforward.
\begin{proof}[Proof of Theorem~\ref{thm.asyden}]
The first part of the Theorem has been already proved in Lemma~\ref{l:8.1} with $\alpha = 1$, keeping in mind~\eqref{e:N8}. For the second part, observe that the joint convergence of $(\wh{\cE}_t, Z_t)$ to $(\wh{\cE}, Z)$ together with the almost sure convergence of $Z_t$ to $Z$, shows that $(\wh{\cE}_t, Z)$ also converges jointly weakly to $(\wh{\cE}, Z)$. Moreover, for any $v \geq 0$ and a Borel set $B \subseteq \bbR $, the map $\wh{\cE} \mapsto \cE\big([-v,\infty) ;\; B\big)$ is continuous in the underlying topology for almost every $\wh{\cE}$. This is because $\wh{\cE}$ has a conditional Poissonian law with a product intensity measure, of which the first coordinate is absolutely continuous with respect to Lebesgue.  

Since $Z$ is almost surely positive, the latter implies that  for all $v \geq 0$,
\begin{equation}
\frac{\cE_t\big([-v,\infty) ;\; [-v, \infty] \big)}{C_\star Zv\rme^{\sqrt{2}v}}
\ \overset{t \to \infty}\Longrightarrow \ 
\frac{\cE \big([-v,\infty) ;\; [-v, \infty]  \big)}{C_\star Zv\rme^{\sqrt{2}v}} \,.
\end{equation}
The numerator on the right hand side is exactly $\cE([-v,\infty))$ in light of~\eqref{e:N8}. For the left hand side, the asymptotic separation of extreme values as manifested in~\eqref{e:108} shows that we can replace the numerator with $\cE_t\big([-v, \infty)\big)$ with the convergence still holding. This together with the first statement of the theorem yields the desired result.
\end{proof}

\begin{proof}[Proof of Theorem~\ref{t:14}]
The first part is again an immediate consequence of Lemma~\ref{l:8.1}. Just divide both numerator and denominator by $C_\star Z v \rme^{\sqrt{2} v}$ for $v \geq 1$, recalling that $Z$ is almost surely positive. Then take $v \to \infty$ and use Lemma~\ref{l:8.1} with the given $\alpha$ for the numerator and $\alpha=1$ for the denominator. Using also relation~\eqref{e:N6}, this gives~\eqref{e:37}.

As for the second part, the same argument as in the previous proof shows that
the numerator and denominator in~\eqref{e:624} converge weakly jointly to the numerator and denominator of~\eqref{e:37} respectively. This together with the first part shows the second part of the theorem. 
\end{proof}

\subsection{Distance to the Second Maximum}
Finally, let us prove the theorem concerning the distance to the second maximum.
\begin{proof}[Proof of Theorem~\ref{t:2.9}]
Starting with the first statement and assuming that $\cE$ is realized as in~\eqref{e:5}, we have
\begin{equation}
\label{e:185}
\big \{ v^1 - v^2 > w \big\} = \big\{u^1 - u^2 > w\big\} \cap \big\{\cC^1([-w, 0)) = 0 \big\} \,.
\end{equation}
Since the cluster decorations are independent of the ``backbone'' Poisson point process $\cE^*$, the two events on the right hand side are independent and hence
\begin{equation}
\label{e:185B}
\bbP \big( v^1 - v^2 > w \big) = \bbP \big(u^1 - u^2 > w\big) \bbP \big(\cC([-w, 0)) = 0 \big) \,.
\end{equation}

Now, to compute the first probability on the right hand side, notice that we can rewrite the intensity measure in the law of $\cE^*$ as $\rme^{-\sqrt{2} (u - (\log Z)/\sqrt{2})} \rmd u$. This recasts $\cE^*$ as a randomly shifted Poisson point process with intensity measure $\rme^{-\sqrt{2}u} \rmd u$.  This random shift is irrelevant for the quantity $u^1 - u^2$ and hence we may even assume that $Z = 1$.

In this case, by conditioning on $u^1$ we can write the probability that $u^1 - u^2 > w$ as 
\begin{multline}
\label{e:M6.25}
\int_{u=-\infty}^\infty \rme^{-\sqrt{2} u -\frac{1}{\sqrt{2}}\rme^{-\sqrt{2} u}}
\exp\Big(- \tfrac{\rme^{-\sqrt{2}u}(\rme^{-\sqrt{2} w} - 1)}{\sqrt{2}} \Big)\rmd u \\
= \rme^{- \sqrt{2} u} \int_{z=-\infty}^\infty \rme^{-\sqrt{2}z -\frac{1}{\sqrt{2}} \rme^{-\sqrt{2}z}} \rmd z \,,
\end{multline}
where we have used the substitution $z=u-w$. The integral converges to a finite positive constant showing that $\bbP\big(u^1 - u^2 > w\big) = C \rme^{-\sqrt{2} w}$.

Therefore, taking the logarithm of both sides in~\eqref{e:185B}, dividing by $w$ and letting $w \to \infty$, the first term converges to $-\sqrt{2}$ in light of what we have just proved, while the second converges to $-2$ in light of Proposition~\ref{prop:right.tail.cluster}. The two together show the first part of the theorem.

For the second part of the theorem, first in light of the tightness of the maximum the joint distribution of the first and second highest points of $\cE_t$ converge weakly to the distribution of $v^1$ and $v^2$. It follows then by the continuous mapping theorem that the distribution of $h_t^* - h_t^{*(2)}$ converges weakly to the distribution of $v^1 - v^2$. This shows~\eqref{e:33} when $w \to \infty$ along continuity points of the distribution of $v^1 - v^2$. The extension to any $w$ follows by monotonicity following arguments similar to the ones used in the proofs before.
\end{proof}

\section*{Acknowledgements} 
The authors would like to thank Anton Bovier and the Institute for Applied Mathematics at Bonn University, as well as Dima Ioffe and the Technion - Israel's institute of technology, for providing a wonderful environment for this collaborative work.

\bibliographystyle{abbrv}
\bibliography{GBBStructure}

\end{document}